\documentclass[11pt]{amsart}

\usepackage{amscd}
\usepackage{amsmath, amssymb}
\usepackage{amsfonts}
\newcommand{\de}{\partial}
\newcommand{\db}{\overline{\partial}}

\newcommand{\ddt}{\frac{\partial}{\partial t}}
\newcommand{\ddb}{\partial \ov{\partial}}
\newcommand{\Tmax}{T_{\textrm{max}}}
\newcommand{\ddbar}{\sqrt{-1} \partial \overline{\partial}}
\newcommand{\Ric}{\mathrm{Ric}}
\newcommand{\ov}[1]{\overline{#1}}
\newcommand{\mn}{\sqrt{-1}}

\newcommand{\tr}[2]{\mathrm{tr}_{#1}{#2}}
\newcommand{\ti}[1]{\tilde{#1}}
\newcommand{\vp}{\varphi}

\newcommand{\ve}{\varepsilon}
\newcommand{\RC}{R^{\textrm{C}}}

\numberwithin{equation}{section}

\renewcommand{\leq}{\leqslant}
\renewcommand{\geq}{\geqslant}
\renewcommand{\le}{\leqslant}
\renewcommand{\ge}{\geqslant}

\begin{document}
\newcounter{remark}
\newcounter{theor}
\setcounter{remark}{0}
\setcounter{theor}{1}
\newtheorem{claim}{Claim}
\newtheorem{theorem}{Theorem}[section]
\newtheorem{lemma}[theorem]{Lemma}
\newtheorem{conj}[theorem]{Conjecture}
\newtheorem{corollary}[theorem]{Corollary}
\newtheorem{proposition}[theorem]{Proposition}
\newtheorem{question}{question}[section]
\newtheorem{defn}{Definition}[theor]

\newenvironment{example}[1][Example]{\addtocounter{remark}{1} \begin{trivlist}
\item[\hskip
\labelsep {\bfseries #1  \thesection.\theremark}]}{\end{trivlist}}

\title[Evolution of a Hermitian metric by its Chern-Ricci form]{On the evolution of a Hermitian metric by its Chern-Ricci form}

\author[V. Tosatti]{Valentino Tosatti$^{*}$}
\thanks{$^{*}$Supported in part by NSF grant DMS-1005457 and by a Blavatnik Award for Young Scientists.
Part of this work was carried out
while the first-named author was visiting the Mathematical Science
Center of Tsinghua University in Beijing, which he would like to
thank for its hospitality.}
\address{Department of Mathematics, Columbia University, 2990 Broadway, New York, NY 10027}
\curraddr{Department of Mathematics, Northwestern University, 2033 Sheridan Road, Evanston, IL 60208}

\author[B. Weinkove]{Ben Weinkove$^{\dagger}$}
\thanks{$^{\dagger}$Supported in part by NSF grant DMS-1105373.}
\address{Mathematics Department, University of California San Diego, 9500 Gilman Drive \#0112, La Jolla, CA 92093}
\curraddr{Department of Mathematics, Northwestern University, 2033 Sheridan Road, Evanston, IL 60208}
\begin{abstract}
We consider the evolution of a Hermitian metric on a compact complex manifold by its Chern-Ricci form.  This is an evolution equation first studied by M. Gill, and coincides with the K\"ahler-Ricci flow if the initial metric is K\"ahler.  We find
 the maximal existence time for the flow in terms of the initial data.  We investigate the behavior of the flow on complex surfaces
 when the initial metric is Gauduchon, on complex manifolds with negative first Chern class, and on some Hopf manifolds.  Finally, we discuss a new estimate for the complex Monge-Amp\`ere equation on Hermitian manifolds.
 \end{abstract}

\maketitle

\maketitle

\section{Introduction}

Let $(M, J)$ be a compact complex manifold of complex dimension $n$.  Let $g_0$ be a Hermitian metric on $M$, that is a  Riemannian metric $g_0$ satisfying $g_0(JX, JY)=g_0(X,Y)$ for all vectors $X$, $Y$.  In local complex coordinates $(z_1, \ldots, z_n)$, the metric $g_0$ is given  by a Hermitian matrix with components $(g_0)_{i \ov{j}}$.  Associated to $g_0$ is a real $(1,1)$ form $\omega_0 = \sqrt{-1} (g_0)_{i \ov{j}} dz_i \wedge d\ov{z}_j$, which we will also often refer to as a Hermitian metric.

Given the success of Hamilton's Ricci flow \cite{Ha} in establishing deep results  in the setting of topological, smooth and Riemannian manifolds (see e.g. \cite{BS, Ha2, P1}),
it is natural to ask whether there is a  parabolic flow of metrics on $M$ which starts at $g_0$, preserves the Hermitian condition and reveals information about the structure of $M$ as a complex manifold.
    In the case when $g_0$ is K\"ahler (meaning $d\omega_0=0$), the Ricci flow does precisely this.  It gives a flow of K\"ahler metrics whose behavior is deeply intertwined with the complex and algebro-geometric properties of $M$ (see \cite{Ca, CW, FIK, PSSW, PS2, ST, ST2, ST3, SW, SW2,   SW4, SY, Sz, Ti, TZ, To1, Ts}, for example).

However, if $g_0$ is not K\"ahler, then in general the Ricci flow  does not preserve the Hermitian condition $g(JX, JY) = g(X,Y)$.  Alternative parabolic flows on complex manifolds which do preserve the Hermitian property have been proposed by Streets-Tian \cite{StT, StT2} and also Liu-Yang \cite{LY}.

This paper is concerned with another such flow, first investigated by M. Gill \cite{G}, which we will call the \emph{Chern-Ricci flow}:
\begin{equation} \label{crf}
\ddt{} \omega = - \Ric(\omega), \quad \omega|_{t=0}=\omega_0,
\end{equation}
where here $\Ric(\omega)$ is the \emph{Chern-Ricci form} (sometimes called the \emph{first Chern form}) associated to the Hermitian metric $g$, which in local coordinates is given by
\begin{equation}
\Ric(\omega) = - \ddbar \log \det g.
\end{equation}
In the case when $g$ is K\"ahler, $\Ric(\omega) = \sqrt{-1} R_{i \ov{j}} dz_i \wedge d\ov{z}_j$, where $R_{i \ov{j}}$ is the usual Ricci curvature of $g$.  Thus if $g_0$ is K\"ahler,  (\ref{crf}) coincides with the K\"ahler-Ricci flow.  In general, $\Ric(\omega)$ does not have a simple relationship with the Ricci curvature of $g$.  The Bott-Chern cohomology class determined by the closed form $\Ric(\omega)$ is denoted by $c^{\textrm{BC}}_1(M)$.  We call this the \emph{first Bott-Chern class} of $M$.  It is independent of the choice of Hermitian metric $\omega$ (see section \ref{section:preliminaries} for more details).

The following result for the Chern-Ricci flow was proved by Gill  \cite{G}:

\begin{theorem}[Gill] \label{theorem:gill}  If $c_1^{\emph{BC}}(M)=0$ then, for any initial metric $\omega_0$, there exists a solution $\omega(t)$ to the Chern-Ricci flow (\ref{crf}) for all time and  the metrics $\omega(t)$ converge smoothly as $t \rightarrow \infty$ to a Hermitian metric $\omega_{\infty}$ satisfying $\emph{Ric}(\omega_{\infty})=0$.
\end{theorem}

Moreover, the Hermitian metric $\omega_{\infty}$ is the unique Chern-Ricci flat metric on $M$ of the form $\omega_{\infty} = \omega_0 + \ddbar \varphi$ for some function $\varphi$.  The Chern-Ricci flat metrics were already known to exist  \cite{Ch, TW2}, and the estimate of \cite{TW2} is used in the proof of Theorem \ref{theorem:gill}.
  If $\omega_0$ is K\"ahler then Theorem \ref{theorem:gill} is due to Cao \cite{Ca}, with $\omega_{\infty}$ being the Ricci-flat metric of Yau \cite{Y}.
In Section \ref{section:preliminaries} below, we discuss the work of Gill \cite{G} further, and also explain
 how the Chern-Ricci flow compares with some other parabolic flows on complex manifolds studied in the literature.

Our first result characterizes the maximal existence time for a solution to the Chern-Ricci flow using information from the initial Hermitian metric $\omega_0$.  First observe that the flow equation (\ref{crf}) may be rewritten as
$$\ddt{} \omega = - \Ric(\omega_0) + \sqrt{-1} \partial \ov{\partial} \theta(t), \quad \textrm{with } \theta (t)= \log \frac{\det g(t)}{\det g_0}.$$
Thus, as long as the flow exists, the solution $\omega(t)$ starting at $\omega_0$ must be of the form $\omega(t) = \alpha_t + \sqrt{-1} \partial \ov{\partial} \Theta$, for some function $\Theta=\Theta(t)$, with
\begin{equation} \label{alphat}
 \alpha_t = \omega_0 - t \Ric(\omega_0).
\end{equation}

  Now define a number
 $T=T(\omega_0)$ with $0< T \le \infty$ by
\begin{equation} \label{T}
T = \sup \{ t \ge 0 \ | \ \exists \psi \in C^{\infty}(M) \textrm{ with }\alpha_t + \ddbar \psi >0 \}.
\end{equation}
By the observation above, a solution to (\ref{crf}) cannot exist beyond time $T$.  We prove:

\begin{theorem} \label{theorem:maximal}
There exists a unique maximal solution to the Chern-Ricci flow (\ref{crf}) on $[0,T)$.
\end{theorem}

In the special case when $\omega_0$ is K\"ahler, this is already known by the result of Tian-Zhang \cite{TZ}, who extended earlier work of Cao and Tsuji \cite{Ca, Ts, Ts2}.  In the K\"ahler case, $T$ depends only on the cohomology class $[\omega_0]$ and can be written $$T= \sup \{ t \ge 0 \ | \ [\omega_0] - t c_1(M) >0 \}.$$
Furthermore the Nakai-Moishezon criterion, due to Buchdahl \cite{Bu2} and Lamari \cite{La2} for K\"ahler surfaces  and to Demailly-P\u{a}un \cite{DP} for general K\"ahler manifolds, implies that at time $T$ either the volume of $M$ goes to zero, or the volume of some proper analytic subvariety of $M$ goes to zero (cf. the discussion in \cite{FIK}).

Note that in the general Hermitian case, we can consider the equivalence relation of $(1,1)$ forms on $M$:
\begin{equation} \nonumber
\alpha \sim \alpha' \quad \Longleftrightarrow \quad \alpha = \alpha' + \ddbar \psi \quad \textrm{for some function } \psi \in C^{\infty}(M).
\end{equation}
Then $T$  defined by (\ref{T}) depends only on the equivalence class of $\omega_0$.

In the special case when $M$ is a complex surface ($n=2$) a result of Gauduchon \cite{Ga} is that every Hermitian metric is conformal to a $\partial \ov{\partial}$-closed metric  $\omega_0$.  If $\omega_0$ is $\partial\ov{\partial}$-closed then so is $\omega(t)$ for $t\in[0,T)$.  Moreover,
 we have a  geometric characterization of the maximal existence time $T$:

\begin{theorem} \label{theorem:surfaces}
Let $M$ be a compact complex surface, $\omega_0$ a $\partial \ov{\partial}$-closed Hermitian metric.  Then $T$ defined by (\ref{T}) can be written as
\[\begin{split}
T=\sup\bigg\{T_0 \geq 0\ \bigg|\ &\forall t\in[0,T_0],\ \int_M \alpha_t^2>0,\quad \int_D \alpha_t >0,\\
&\mathrm{for\ all\ } D \mathrm{\ irreducible\ effective\ divisors\ with\ }D^2<0\bigg\},
\end{split}\]
for $\alpha_t$ given by (\ref{alphat}).
\end{theorem}

Note that for  $t\in [0,T)$, the quantity $\int_M \alpha_t^2 = \int_M \omega(t)^2$ is the volume of $M$ (with respect to $\omega(t)$) and $\int_D \alpha_t = \int_D \omega(t)$ is the volume of the curve $D$.   Thus we can restate Theorem \ref{theorem:surfaces} as:

\begin{corollary} \label{corollary:surfaces}
Let $M$ be a compact complex surface, $\omega_0$ a $\partial \ov{\partial}$-closed Hermitian metric. Then the Chern-Ricci flow (\ref{crf}) starting at $\omega_0$ exists until either the volume of $M$ goes to zero, or the volume of a curve of negative self-intersection goes to zero.
\end{corollary}

As we remarked above, the same result was known to hold for the K\"ahler-Ricci flow thanks to the Nakai-Moishezon criterion of \cite{Bu2, La2}.  Analogues of Theorems \ref{theorem:maximal}, \ref{theorem:surfaces} and Corollary \ref{corollary:surfaces} were conjectured by Streets-Tian \cite{StT3} for their pluriclosed flow (see Section \ref{section:preliminaries} below).

The K\"ahler-Ricci flow has a deep connection to the Minimal Model Program in algebraic geometry, as demonstrated by the work of Song-Tian and others \cite{EGZ, LT,  ST, ST2, ST3,  SW, SW2, SW3, SW4, SY, Ti, TZ, Ts, Z}.
In the case of algebraic surfaces, the minimal model program is relatively simple.  Indeed, a \emph{minimal surface} is defined to be a surface with no  $(-1)$-curves (smooth rational curves $C$ with $C^2 =-1$).  To find the minimal model, one can just apply a finite number of \emph{blow-downs}, which are algebraic operations contracting the $(-1)$-curves.    It was shown in \cite{SW2} that the K\"ahler-Ricci flow on an algebraic surface carries out these algebraic operations, contracting $(-1)$-curves in the sense of Gromov-Hausdorff, and smoothly outside of the curves.
Moreover, the same behavior occurs on a non-algebraic K\"ahler surface \cite{SW4}. In all dimensions, weak solutions to the K\"ahler-Ricci flow through singularites were constructed in \cite{ST3} and a number of conjectures were made about the metric behavior of the flow (see also \cite{SY}).

We return now to the case of a complex (non-K\"ahler) surface.
In this case one can also contract the $(-1)$ curves to arrive at a minimal surface.  We conjecture that the Chern-Ricci flow on a complex surface starting at a $\partial \ov{\partial}$-closed metric behaves in an analogous way to the K\"ahler-Ricci flow on a K\"ahler surface.   We prove the following:

\begin{theorem} \label{theorem:surfaces2} Let $M$ be a compact complex surface with a $\partial \ov{\partial}$-closed Hermitian metric $\omega_0$, and let $[0,T)$ be the maximal existence time of the Chern-Ricci flow starting from $\omega_0$. Then
\begin{enumerate}
\item[(a)] If $T=\infty$ then $M$ is minimal
\item[(b)] If $T<\infty$ and $\mathrm{Vol}(M,\omega(t))\to 0$ as $t\to T^-$, then $M$ is either birational to a ruled surface or it is a surface of class $VII$ (and in this case it cannot be an Inoue surface)
\item[(c)] If $T<\infty$ and $\mathrm{Vol}(M,\omega(t))$ stays positive as $t\to T^-$, then $M$ contains $(-1)$-curves.
\end{enumerate}
Furthermore, if $M$ is minimal then $T=\infty$ unless $M$ is $\mathbb{CP}^2$, a ruled surface, a Hopf surface or a surface of class VII with $b_2>0$, in which cases (b) holds.
\end{theorem}

In the case that $M$ is not minimal, and (c) occurs, we expect that the Chern-Ricci flow will contract a finite number of $(-1)$-curves and can be uniquely continued on the new manifold.
 Moreover, we conjecture that this process can be repeated until one obtains a minimal surface, or ends up in case (b) above.  More details of this conjecture can be found in Section \ref{section:surfaces}.

To provide some evidence for our conjecture, we prove the following theorem.  It is an analogue of a result for the K\"ahler-Ricci flow, whose proof is essentially contained in \cite{TZ} (for a recent exposition, see Chapter 7 of  \cite{SW4}), and
which was a key starting point for the work \cite{ST3, SW, SW2, SW3}.  We assume that the maximal existence time $T$ is finite and, roughly speaking, that the limiting `class' of the flow at time $T$  is given by the pull-back of a Hermitian metric  on a manifold $N$ via $\pi : M \rightarrow N$, where $\pi$ is a holomorphic map blowing down an exceptional divisor $E$ to a point $p \in N$.  We show that the solution to the Chern-Ricci flow will converge smoothly at time $T$ away from $E$.  In this way, one can obtain a Hermitian metric on the new manifold $N$, at least away from the point $p$.  Our result holds in any dimension:

\begin{theorem} \label{theorem:blowdown} Assume that there exists a holomorphic map between compact Hermitian manifolds $\pi : (M, \omega_0) \rightarrow (N, \omega_N)$ blowing down the exceptional divisor $E$ on $M$ to a point $p\in N$.  In addition, assume that there exists a smooth function $\psi$ on $M$ with
\begin{equation}
\omega_0 - T \emph{Ric}\, (\omega_0) + \ddbar \psi = \pi^* \omega_N,
\end{equation}
with $T<\infty$ given by (\ref{T}).

Then the solution $\omega(t)$ to the Chern-Ricci flow (\ref{crf}) starting at $\omega_0$ converges in $C^{\infty}$ on compact subsets of $M \setminus E$ to a smooth Hermitian metric $\omega_{T}$ on $M \setminus E$.
\end{theorem}

There are some new obstacles to proving this in the non-K\"ahler case that we overcome using  a parabolic Schwarz Lemma for volume forms, and a second order estimate for the metric which uses a trick of Phong-Sturm \cite{PS}.

In the cases when the flow has a long time solution, it is natural to investigate its behavior at infinity.
If the manifold has vanishing first Bott-Chern class, we have already seen by Gill's result (Theorem \ref{theorem:gill})
 that the flow converges to a Chern-Ricci flat Hermitian metric.  We now suppose that the first Chern class $c_1(M)$ is \emph{negative} (note that $c_1^{\textrm{BC}}(M)<0$ implies    $c_1(M)<0$).
In this case,
the manifold $M$ is K\"ahler and a fundamental result of Aubin \cite{Au} and Yau \cite{Y} says that $M$ admits a unique
K\"ahler-Einstein metric $\omega_{\mathrm{KE}}$ with negative scalar curvature. Cao \cite{Ca} then proved that the  K\"ahler-Ricci flow (appropriately normalized) deforms any K\"ahler metric in $-c_1(M)$ to $\omega_{\mathrm{KE}}$.  The same is true for the normalized K\"ahler-Ricci flow starting at any K\"ahler metric \cite{TZ, Ts}.
Our next result shows that starting at  any \emph{Hermitian metric} on the manifold $M$, the (normalized) Chern-Ricci flow will converge to the K\"ahler-Einstein metric $\omega_{\mathrm{KE}}$.

\begin{theorem}\label{theorem:c1n}
Let $M$ be a compact complex manifold with $c_1(M)<0$ and let $\omega_0$ be
a Hermitian metric on $M$. Then the Chern-Ricci flow \eqref{crf}
has a long-time solution $\omega(t)$, and as $t$ goes to infinity the rescaled metrics $\omega(t)/{t}$
converge smoothly to the unique K\"ahler-Einstein metric $\omega_{\mathrm{KE}}$ on $M$.
\end{theorem}

In particular we see that  the Chern-Ricci flow on these manifolds, after normalization, deforms any Hermitian metric to a K\"ahler one.

Next we illustrate the Chern-Ricci flow with an explicit example.
 For $\alpha=(\alpha_1, \ldots, \alpha_n) \in \mathbb{C}^n \setminus \{0 \}$ with $|\alpha_1| = \cdots = |\alpha_n| \neq 1$, we consider the
 Hopf manifold $M_{\alpha} = (\mathbb{C}^n \setminus \{0 \})/\sim$ where
 \begin{equation} \nonumber
 (z_1, \ldots, z_n) \sim  (\alpha_1 z_1, \ldots, \alpha_n z_n).
\end{equation}
This is a non-K\"ahler complex manifold of complex dimension $n$.  If $n=2$, it is an example of a class VII surface.  We can write down an exact solution to the Chern-Ricci flow on $M_{\alpha}$.
Consider the metric
$\omega_H = \frac{\delta_{ij}}{r^2} \mn dz_i \wedge d\ov{z}_j.$
Then we have:

\begin{proposition} \label{prop:hopf1}
The metrics  $\omega(t) := \omega_H - t \emph{Ric}\, (\omega_H)$ on $M_{\alpha}$ give a solution of the Chern-Ricci flow on the maximal existence interval $[0,1/n)$.  As $t\rightarrow T=1/n$, the limiting  nonnegative $(1,1)$ form $\omega_{T}$ is given by
\begin{equation} \nonumber
\omega_{T} =  \frac{\overline{z}_i z_j}{r^4} \sqrt{-1} dz_i\wedge d\ov{z}_j.
\end{equation}
\end{proposition}

In the case of
  the original Hopf surface, which has $\alpha = (2,2)$ and is an elliptic fiber bundle over $\mathbb{P}^1$ via the map $(z_1, z_2) \mapsto [z_1, z_2]$, the limiting form $\omega_T$ is positive definite along the fibers and zero in directions orthogonal to the fibers.

If we start with any metric $\omega_0$ which differs from $\omega_H$ by $\ddbar \psi$ for some function $\psi$, then we conjecture that the flow also converges as $t \rightarrow T$ to a smooth but degenerate $(1,1)$ form on $M_{\alpha}$ with properties similar to $\omega_T$.  To give evidence for this conjecture, we prove an estimate:

\begin{proposition} \label{prop:hopf2} Let $\omega_0 = \omega_H + \ddbar \psi$ be a Hermitian metric on $M_{\alpha}$, and let $\omega(t)$ be the solution of the Chern-Ricci flow (\ref{crf}) starting at $\omega_0$ on $M_{\alpha}$ for $t \in [0,1/n)$.
Then there exists a uniform constant $C$ such that
\begin{equation} \nonumber
 \omega(t) \le C \omega_H, \quad \emph{for } t \in [0,1/n).
\end{equation}
\end{proposition}

In particular, this result shows that we obtain convergence for the flow at the level of potential functions in $C^{1+\beta}$ for any $\beta \in (0,1)$.  For more details see Section \ref{section:hopf}.

The final result in this paper concerns not the Chern-Ricci flow, but an elliptic equation:  the complex Monge-Amp\`ere equation
\begin{equation}\label{maell}
(\omega + \ddbar \varphi)^n = e^F \omega^n, \quad \omega':=\omega+ \ddbar \varphi>0,
\end{equation}
on a compact Hermitian manifold $(M, \omega)$, where  $F$ is a smooth function on $M$.  We give an alternative proof of a result of \cite{TW2} that $\| \varphi \|_{C^0}$ is uniformly bounded (see also \cite{Bl, DK}).  The result makes use of a new second order estimate in this context: $\tr{\omega}{\omega'} \le Ce^{A (\varphi- \inf_M \varphi)}$ which we conjectured to hold in \cite{TW1}.  For more details see Section \ref{section:monge}.  We have included this result here because it follows easily from the argument used in Theorem \ref{theorem:blowdown} together with the method of \cite{TW1}.  The key new ingredient is the trick of Phong-Sturm \cite{PS} applied to this setting.\\

\section{Preliminaries and comparison with other flows} \label{section:preliminaries}

In this section, we include for the reader's convenience some background material on local coordinate computations with Hermitian metrics.

Let $(M, g)$ be a compact Hermitian manifold of complex dimension $n$.  We will often compute in complex coordinates $z_1, \ldots, z_n$.  In this case, $g$ is determined by the $n \times n$ Hermitian matrix $g_{i \ov{j}} = g( \partial_i, \partial_{\ov{j}})$, where we are writing $\partial_i, \partial_{\ov{j}}$ for $\frac{\partial}{\partial z_i}, \frac{\de}{\de \ov{z}_j}$ respectively.  We denote by $g^{\ov{j}i}$ the entries of the inverse matrix of $(g_{i\ov{j}})$.

We define the \emph{Chern connection} $\nabla$ associated to $g$ as follows.  Given a vector field $X= X^i \partial_i$ and a $(1,0)$ form $a=a_i dz_i$, we define $\nabla X$ and $\nabla a$ to be the tensors with components:
\begin{equation} \nonumber
\nabla_i X^k = \partial_i X^k + \Gamma^k_{ij} X^j, \quad \nabla_i a_j = \partial_i a_j - \Gamma_{ij}^k a_k,
\end{equation}
where the Christoffel symbols $\Gamma^k_{ij}$ are given by
\begin{equation} \nonumber
\Gamma^k_{ij} = g^{\ov{q} k} \partial_i g_{j \ov{q}}.
\end{equation}
The tensors $\nabla \ov{X}$ and $\nabla \ov{a}$ have components
$\nabla_{i} \ov{X^{k}} = \partial_i \ov{X^k}$ and  $\nabla_i \ov{a_j} = \partial_i \ov{a_j}$.  The connection $\nabla$ can be naturally extended to any kind of tensor, and we have $\nabla_k g_{i \ov{j}}=0$.

We write $\Delta$ for the \emph{complex Laplacian} of $g$, which acts on a function $f$ by
\begin{equation} \nonumber
\Delta f = g^{\ov{j}i} \partial_i \partial_{\ov{j}} f = g^{\ov{j}i} \nabla_i \nabla_{\ov{j}} f.
\end{equation}
The \emph{torsion} of $g$ is the tensor $T$ with components
\begin{equation} \nonumber
T^k_{ij} = \Gamma^k_{ij} - \Gamma^k_{ji}.
\end{equation}
The torsion tensor vanishes in the special case that $g$ is K\"ahler.

We define the \emph{curvature} of $g$ to be the tensor with components
\begin{equation} \nonumber
R_{k \ov{\ell} i}^{\ \ \ \, p} = - \partial_{\ov{\ell}} \Gamma^p_{ki}.
\end{equation}
We will often raise and lower indices using the metric $g$, writing for example $R_{k \ov{\ell} i \ov{j}} = g_{p \ov{j}} R_{k \ov{\ell} i}^{\ \ \ \, p}$.  Note that $\ov{R_{k \ov{\ell} i \ov{j}}} = R_{\ell \ov{k} j \ov{i}}$.  We have the following commutation formulae:
\begin{align} \nonumber
[\nabla_k, \nabla_{\ov{\ell}}] X^i & = R_{k \ov{\ell} j}^{\ \ \ \,  i} X^j, \quad [\nabla_k, \nabla_{\ov{\ell}}] \ov{X^i} = - R_{k \ov{\ell} \ \, \ov{j}}^{\ \ \, \, \ov{i}} \, \ov{X^j}, \\ \nonumber
[\nabla_k, \nabla_{\ov{\ell}}] a_j & = - R_{k \ov{\ell} j}^{\ \ \ \, i} a_i, \quad [\nabla_k, \nabla_{\ov{\ell}}] \ov{a_j} = R_{k \ov{\ell} \ \, \ov{j}}^{\ \ \, \, \ov{i}} \, \ov{a_i},
\end{align}
where we are writing $[\nabla_k, \nabla_{\ov{\ell}}]$ for $\nabla_k \nabla_{\ov{\ell}} - \nabla_{\ov{\ell}} \nabla_k$.  We write the \emph{Chern-Ricci curvature} of $g$ as the tensor $\RC_{k \ov{\ell}}$ given by
\begin{equation} \nonumber
\RC_{k \ov{\ell}} = g^{\ov{j} i} R_{k \ov{\ell} i \ov{j}} = - \partial_k \partial_{\ov{\ell}} \log \det g,
\end{equation}
so that the Chern-Ricci form is equal to
$$\Ric(\omega)=\mn \RC_{k \ov{\ell}} dz_k\wedge d\ov{z}_\ell.$$
It is a closed real $(1,1)$ form and its cohomology class in the Bott-Chern cohomology group
$$H^{1,1}_{\mathrm{BC}}(M, \mathbb{R})=\frac{\{\textrm{closed real }(1,1)\textrm{ forms}\}}{\{\ddbar \psi, \psi\in C^\infty(M,\mathbb{R})\}},$$
is the first Bott-Chern class of $M$, and is denoted by $c_1^{\mathrm{BC}}(M)$. It  is independent of the choice of Hermitian metric $\omega$. More generally, if $\Omega$ is a smooth positive volume form on $M$ we can define locally
$\Ric(\Omega)=-\ddbar\log\Omega$, which is a global closed real $(1,1)$ form that represents $c_1^{\mathrm{BC}}(M)$.  For notational convenience, we omit the factor of $2\pi$ that usually appears in the definition of $c_1^{\mathrm{BC}}(M)$.  The downside of this convention is that some factors of $2\pi$ will appear later in the cohomological calculations of Section \ref{section:surfaces}.

We end this section by briefly mentioning some related parabolic equations on Hermitian manifolds which have previously been studied in the literature.
Streets-Tian \cite{StT2} introduced the flow
\begin{equation} \label{HCF}
\ddt{} g_{i \ov{j}} = - S_{i\ov{j}} +Q_{i \ov{j}},\quad g|_{t=0} = g_0,
\end{equation}
where $S_{i\ov{j}}$ is given by taking `the other trace' of the curvature of the Chern connection:
$$S_{i \ov{j}} = g^{\ov{\ell} k} R_{k \ov{\ell} i \ov{j}},$$
and $Q_{i \ov{j}}$ is a certain quadratic term in the torsion.  If the form $\omega_0$ associated to $g_0$ satisfies $\partial \overline{\partial} \omega_0=0$, this equation becomes their \emph{pluriclosed flow}
\begin{equation} \label{pluriclosed}
\ddt{} \omega = \partial \partial^* \omega + \ov{\partial} \, \ov{\partial}^*  \omega- \Ric (\omega), \quad \omega|_{t=0} = \omega_0,
\end{equation}
and if $g_0$ is K\"ahler, it coincides with the K\"ahler-Ricci flow.  They analyzed (\ref{pluriclosed}) in detail in \cite{StT, StT3} and made a number of conjectures about it, two of which are analogues of our Theorems \ref{theorem:maximal} and \ref{theorem:surfaces}.  They conjecture that their flow can be used to study the topology of class VII$^+$ surfaces.  In addition, Streets-Tian considered a family of flows of the form (\ref{HCF}) with arbitrary quadratic torsion term $Q$ and proved, among other results, a short-time existence theorem \cite{StT2}. The flow (\ref{HCF}) was extended to the almost complex setting by Vezzoni \cite{V}.  Liu-Yang \cite{LY} propose studying  the flow (\ref{HCF}) in the  case of $Q=0$.

In \cite{G}, Gill introduced the following parabolic complex Monge-Amp\`ere equation on a compact Hermitian manifold $(M, \hat{g})$:
\begin{equation} \label{gill}
\ddt{\varphi} = \log \frac{\det (\hat{g}_{i \ov{j}} + \partial_i \partial_{\ov{j}} \varphi)}{\det (\hat{g}_{i \ov{j}})} -F, \quad \hat{g}_{i \ov{j}} + \partial_i \partial_{\ov{j}} \varphi>0, \quad \varphi|_{t=0} = 0,
\end{equation}
for a fixed smooth function $F$ on $M$.  He showed that the unique solution to (\ref{gill}) exists for all time and, after an appropriate normalization, converges in $C^{\infty}$ to a smooth function $\varphi_{\infty}$ solving the complex Monge-Amp\`ere equation
\begin{equation} \label{cma}
\log \frac{\det (\hat{g}_{i \ov{j}} + \partial_i \partial_{\ov{j}} \varphi_{\infty})}{\det (\hat{g}_{i \ov{j}})} = F+b,
\end{equation}
for a constant $b$ which is uniquely determined.  The existence of solutions to the elliptic equation (\ref{cma}) on Hermitian manifolds (generalizing Yau's Theorem \cite{Y}) was already known by the work of Cherrier \cite{Ch} (if $n=2$) and the authors \cite{TW2} ($n>2$).   See also \cite{GL, ZX}.  In the special case where $\hat{g}$ is K\"ahler, the flow (\ref{gill}) had been considered earlier by Cao \cite{Ca}, who proved the analogous results.

In the case when $c_1^{\textrm{BC}}(M)=0$, we can find a function $F$ satisfying $$ \partial \ov{\partial} \log \det \hat{g} = \partial \ov{\partial} F,$$ and with this choice, $\omega(t)=\hat{\omega} + \sqrt{-1}\partial \ov{\partial} \varphi(t)$ for $\varphi(t)$ solving (\ref{gill}) is exactly the Chern-Ricci flow starting at $\hat{\omega}$.
 In general, the only difference between the Chern-Ricci flow and Gill's flow (\ref{gill}) is that for the Chern-Ricci flow we replace the fixed metric $\hat{g}$ by a smoothly varying family of Hermitian metrics $\hat{g}_t$, and replace $F$ by a particular function, which may also depend on $t$.   Many of Gill's estimates carry over easily to the case of the Chern-Ricci flow and we will make extensive use of them here.

A final remark about notation. In the following, $C,C'$ will denote uniform positive constants which may vary from line to line.

\section{Evolution of the trace of the metric}

In this section we write down a formula for the evolution of the trace of the evolving metric with respect to a fixed Hermitian metric.  We will need this calculation in later sections.  We carry out the computation here using tensorial quantities, following \cite{Ch}, rather than using a particular choice of complex coordinates as in \cite{G, GL, StT2, TW1}, for example.

We suppose that we have three Hermitian metrics
$g,g_0$ and $\hat{g}$ such that $g=g(t)$  satisfies the Chern-Ricci flow (\ref{crf}),
and such that the corresponding forms satisfy
\begin{equation} \label{plusclosed}
\omega=\omega_0+ \eta(t),
\end{equation}
for a closed $(1,1)$ form $\eta(t)$.

We denote by $\hat{\nabla}$, $\hat{T}$, $\hat{\Gamma}$, $\hat{R}$ the Chern connection, torsion, Christoffel symbols and curvature of $\hat{g}$. Denote by $T_0$ the torsion tensor of $g_0$ and by $\Delta$ the complex Laplacian associated to $g=g(t)$.  

Note that for the purposes of this paper we will in fact only need the case of $\hat{g}=g_0$.  However, we included the more general calculation below since we anticipate that it may be useful in the future.

We have:

\begin{proposition} \label{prop:bigcalc} The evolution of $\log \emph{tr}_{\hat{g}}g$ is given by
\begin{eqnarray}\label{long}
\left(\ddt{}-\Delta\right) \log\tr{\hat{g}}g&=& (I) + (II) + (III)
\end{eqnarray}
where
\begin{align*} \nonumber
 (I) & = \frac{1}{\tr{\hat{g}}g} \bigg[-g^{\ov{j}p}g^{\ov{q}i}\hat{g}^{\ov{\ell}k}\hat{\nabla}_k g_{i\ov{j}}\hat{\nabla}_{\ov{\ell}}g_{p\ov{q}} +\frac{1}{\tr{\hat{g}}g} g^{\ov{\ell}k}\hat{\nabla}_k \tr{\hat{g}}{g}\hat{\nabla}_{\ov{\ell}}\tr{\hat{g}}{g} \\ & \mbox{}\ \ \ \
- 2\mathrm{Re}\left(g^{\ov{j}i}\hat{g}^{\ov{\ell}k} \hat{T}^p_{ki}\hat{\nabla}_{\ov{\ell}}g_{p\ov{j}}\right) -g^{\ov{j}i}\hat{g}^{\ov{\ell}k}\hat{T}^p_{ik}\ov{\hat{T}^q_{j\ell}}g_{p\ov{q}}\bigg] \\
(II) & = \frac{1}{\tr{\hat{g}}g} \bigg[ g^{\ov{j}i}\hat{g}^{\ov{\ell}k}(\hat{\nabla}_i\ov{\hat{T}^q_{j\ell}}-\hat{R}_{i\ov{\ell}p\ov{j}}\hat{g}^{\ov{q}p})g_{k\ov{q}}
\bigg] \\
(III) & = - \frac{1}{\tr{\hat{g}}g} \bigg[g^{\ov{j}i}\hat{g}^{\ov{\ell}k}\left(\hat{\nabla}_i  \left(  \ov{(T_0)^p_{j\ell}} (g_0)_{k \ov{p}} \right) +
 \hat{\nabla}_{\ov{\ell}} \left( (T_0)^p_{ik} (g_0)_{p \ov{j}} \right) \right)\\
 &\ \ \  -g^{\ov{j}i}\hat{g}^{\ov{\ell}k}\ov{\hat{T}^q_{j\ell}} (T_0)^p_{ik} (g_0)_{p \ov{q}}
\bigg].
\end{align*}
Moreover we have
\begin{align} \nonumber
(I) & \leq  \frac{2}{(\emph{tr}_{\hat{g}}g)^2}\mathrm{Re}\left(
\hat{g}^{\ov{\ell}i}g^{\ov{q}k} (T_0)^p_{ki}(g_0)_{p\ov{\ell}} \hat{\nabla}_{\ov{q}} \emph{tr}_{\hat{g}}g\right), \\ \nonumber
(II) & \leq C \emph{tr}_g \hat{g},
\end{align}
for a constant $C$ that depends only on $\hat{g}$.
If we are at a point where $\emph{tr}_{\hat{g}} g \geq 1$, then
\begin{equation} \nonumber
(III)  \leq C' \emph{tr}_g \hat{g},
\end{equation}
for $C'$ depending only on $g_0$ and $\hat{g}$.
\end{proposition}

\begin{proof}
First,
$$\Delta \tr{\hat{g}}g=g^{\ov{j}i} \hat{\nabla}_i \hat{\nabla}_{\ov{j}}(\hat{g}^{\ov{\ell}k}g_{k\ov{\ell}})=g^{\ov{j}i}\hat{g}^{\ov{\ell}k}\hat{\nabla}_i \hat{\nabla}_{\ov{j}} g_{k\ov{\ell}}.$$
From the definition of covariant derivative,
$$\hat{\nabla}_{\ov{j}} g_{k\ov{\ell}}=\de_{\ov{j}}g_{k\ov{\ell}}-\ov{\hat{\Gamma}^p_{j\ell}}
g_{k\ov{p}},$$
and skew-symmetrizing in $j,\ell$
$$\hat{\nabla}_{\ov{j}} g_{k\ov{\ell}}=\hat{\nabla}_{\ov{\ell}} g_{k\ov{j}}+(\db\omega)_{\ov{j}k\ov{\ell}}-\ov{\hat{T}^p_{j\ell}}g_{k\ov{p}}.$$
But from (\ref{plusclosed}), $\db\omega=\db\omega_0$, and we may rewrite this in terms of  the torsion of $g_0$ as $(\db\omega_0)_{\ov{j}k\ov{\ell}}= \ov{(T_0)^p_{j\ell}}(g_0)_{k\ov{p}}$.
Thus
$$\hat{\nabla}_i\hat{\nabla}_{\ov{j}}g_{k\ov{\ell}} = \hat{\nabla}_i\hat{\nabla}_{\ov{\ell}} g_{k\ov{j}}+\hat{\nabla}_i \left( \ov{(T_0)^p_{j\ell}}(g_0)_{k\ov{p}}\right)- (\hat{\nabla}_i  \ov{\hat{T}^q_{j\ell}} ) g_{k\ov{q}} -\ov{\hat{T}^q_{j\ell}}\hat{\nabla}_ig_{k\ov{q}}.$$
Switching covariant derivatives
$$\hat{\nabla}_i \hat{\nabla}_{\ov{\ell}} g_{k\ov{j}}= \hat{\nabla}_{\ov{\ell}} \hat{\nabla}_i g_{k\ov{j}}-
\hat{R}_{i\ov{\ell}k\ov{q}}\hat{g}^{\ov{q}p}g_{p\ov{j}}+
\hat{R}_{i\ov{\ell}p\ov{j}}\hat{g}^{\ov{q}p}g_{k\ov{q}}.$$
Arguing as above,
$$\hat{\nabla}_{\ov{\ell}}\hat{\nabla}_i g_{k\ov{j}}=\hat{\nabla}_{\ov{\ell}}\hat{\nabla}_k g_{i\ov{j}}+ \hat{\nabla}_{\ov{\ell}} \left( (T_0)^p_{ik} (g_0)_{p \ov{j}} \right)
-(\hat{\nabla}_{\ov{\ell}} \hat{T}^p_{ik})g_{p\ov{j}}- \hat{T}^p_{ik}\hat{\nabla}_{\ov{\ell}}g_{p\ov{j}}.$$
Combining all of these we have
\begin{equation}\label{laplac}
\begin{split}
\Delta \tr{\hat{g}}g&=g^{\ov{j}i}\hat{g}^{\ov{\ell}k}\hat{\nabla}_{\ov{\ell}}\hat{\nabla}_k g_{i\ov{j}}
+g^{\ov{j}i}\hat{g}^{\ov{\ell}k}\bigg( \hat{\nabla}_i \left( \ov{(T_0)^p_{j\ell}}(g_0)_{k\ov{p}}\right)\\
&\ \ \ +\hat{\nabla}_{\ov{\ell}} \left( (T_0)^p_{ik} (g_0)_{p \ov{j}} \right) -(\hat{\nabla}_i\ov{\hat{T}^q_{j\ell}}
-\hat{R}_{i\ov{\ell}p\ov{j}}\hat{g}^{\ov{q}p})g_{k\ov{q}}\\
&\ \ \ -(\hat{\nabla}_{\ov{\ell}}\hat{T}^p_{ik}+\hat{R}_{i\ov{\ell}k\ov{q}}\hat{g}^{\ov{q}p})g_{p\ov{j}}
-\ov{\hat{T}^q_{j\ell}}\hat{\nabla}_ig_{k\ov{q}}-\hat{T}^p_{ik}\hat{\nabla}_{\ov{\ell}}g_{p\ov{j}}\bigg).
\end{split}\end{equation}
We will make a change to the second to last term using
\begin{equation}\label{laplac2}
\ov{\hat{T}^q_{j\ell}}\hat{\nabla}_ig_{k\ov{q}}=\ov{\hat{T}^q_{j\ell}}\hat{\nabla}_kg_{i\ov{q}}
+ \ov{\hat{T}^q_{j\ell}}  (T_0)^p_{ik} (g_0)_{p \ov{q}}
-\hat{T}^p_{ik}\ov{\hat{T}^q_{j\ell}}g_{p\ov{q}}.
\end{equation}
On the other hand,
$$\ddt{} \tr{\hat{g}}g=\hat{g}^{\ov{\ell}k}\de_k\de_{\ov{\ell}}
\log\det(g)=g^{\ov{j}i}\hat{g}^{\ov{\ell}k}\de_k\de_{\ov{\ell}} g_{i\ov{j}}
-g^{\ov{j}p}g^{\ov{q}i}\hat{g}^{\ov{\ell}k}\de_k g_{i\ov{j}}\de_{\ov{\ell}}g_{p\ov{q}},$$
and we wish to convert the partial derivatives into covariant ones. For this, we use the relations
$$\de_k g_{i\ov{j}}=\hat{\nabla}_k g_{i\ov{j}}+\hat{\Gamma}_{ki}^r g_{r\ov{j}}$$
and
\[\begin{split}
\de_{\ov{\ell}}\de_k g_{i\ov{j}}&=\de_{\ov{\ell}}\hat{\nabla}_k g_{i\ov{j}}+\left(\de_{\ov{\ell}}\hat{\Gamma}_{ki}^r \right) g_{r\ov{j}}+\hat{\Gamma}_{ki}^r \de_{\ov{\ell}}g_{r\ov{j}}\\
&=\hat{\nabla}_{\ov{\ell}}\hat{\nabla}_k g_{i\ov{j}}+\ov{\hat{\Gamma}^s_{\ell j}}\hat{\nabla}_k g_{i\ov{s}}-\hat{R}_{k\ov{\ell} i\ov{q}}\hat{g}^{\ov{q}r} g_{r\ov{j}}+\hat{\Gamma}_{ki}^r \hat{\nabla}_{\ov{\ell}}g_{r\ov{j}}+\hat{\Gamma}_{ki}^r \ov{\hat{\Gamma}^s_{\ell j}}g_{r\ov{s}}.
\end{split}\]
Substituting we get
\begin{equation} \label{ddttrace}
\ddt{} \tr{\hat{g}}g=g^{\ov{j}i}\hat{g}^{\ov{\ell}k}\hat{\nabla}_{\ov{\ell}}\hat{\nabla}_k g_{i\ov{j}}
-g^{\ov{j}p}g^{\ov{q}i}\hat{g}^{\ov{\ell}k}\hat{\nabla}_k g_{i\ov{j}}\hat{\nabla}_{\ov{\ell}}g_{p\ov{q}}
-\hat{g}^{\ov{\ell}k}\hat{g}^{\ov{j}i} \hat{R}_{k\ov{\ell} i\ov{j}}
\end{equation}
and so, combining (\ref{laplac}), (\ref{laplac2}) and (\ref{ddttrace}),
\[\begin{split}
&\left(\ddt{}-\Delta \right) \log\tr{\hat{g}}g\\
&=\frac{1}{\tr{\hat{g}}g}\Bigg(\bigg[-g^{\ov{j}p}g^{\ov{q}i}\hat{g}^{\ov{\ell}k}\hat{\nabla}_k g_{i\ov{j}}\hat{\nabla}_{\ov{\ell}}g_{p\ov{q}}+\frac{1}{\tr{\hat{g}}g} g^{\ov{\ell}k}\hat{\nabla}_k \tr{\hat{g}}{g}
\hat{\nabla}_{\ov{\ell}}\tr{\hat{g}}{g}\\
&\ \ \ - 2\mathrm{Re}\left(g^{\ov{j}i}\hat{g}^{\ov{\ell}k} \hat{T}^p_{ki}\hat{\nabla}_{\ov{\ell}}g_{p\ov{j}}\right)
-g^{\ov{j}i}\hat{g}^{\ov{\ell}k}\hat{T}^p_{ik}\ov{\hat{T}^q_{j\ell}}g_{p\ov{q}}\bigg]\\
&\ \ \ +\bigg[g^{\ov{j}i}\hat{g}^{\ov{\ell}k}(\hat{\nabla}_i\ov{\hat{T}^q_{j\ell}}
-\hat{R}_{i\ov{\ell}p\ov{j}}\hat{g}^{\ov{q}p})g_{k\ov{q}}
+\hat{g}^{\ov{\ell}k}(\hat{\nabla}_{\ov{\ell}}\hat{T}^i_{ik}+ \hat{R}_{i\ov{\ell}k\ov{q}}\hat{g}^{\ov{q}i}
- \hat{R}_{k\ov{\ell} i\ov{q}}\hat{g}^{\ov{q}i})\bigg]\\
&\ \ \ -\bigg[g^{\ov{j}i}\hat{g}^{\ov{\ell}k}\left(\hat{\nabla}_i \left( \ov{(T_0)^p_{j\ell}}(g_0)_{k\ov{p}}\right) +\hat{\nabla}_{\ov{\ell}} \left( (T_0)^p_{ik} (g_0)_{p \ov{j}} \right)\right)\\
&\ \ \ -g^{\ov{j}i}\hat{g}^{\ov{\ell}k}\ov{\hat{T}^q_{j\ell}} (T_0)^p_{ik} (g_0)_{p\ov{q}}\bigg] \Bigg).
\end{split}\]
Note that in the first set of square brackets above we have used the identity $\hat{T}^p_{ki} = -\hat{T}^p_{ik}$.
To obtain (\ref{long}) it remains to show that the expression $$\hat{g}^{\ov{\ell}k}(\hat{\nabla}_{\ov{\ell}}\hat{T}^i_{ik}+ \hat{R}_{i\ov{\ell}k\ov{q}}\hat{g}^{\ov{q}i}
- \hat{R}_{k\ov{\ell} i\ov{q}}\hat{g}^{\ov{q}i})$$ in the second set of square brackets vanishes.
 To see this, note that
$$\hat{R}_{i\ov{\ell}k\ov{q}}-\hat{R}_{k\ov{\ell}i\ov{q}}=-\hat{g}_{j\ov{q}}\de_{\ov{\ell}}\hat{T}^j_{ik}=-\hat{g}_{j\ov{q}}\hat{\nabla}_{\ov{\ell}}\hat{T}^j_{ik},$$
and so
$$\hat{g}^{\ov{\ell}k}\hat{g}^{\ov{q}i}(\hat{R}_{i\ov{\ell}k\ov{q}}
- \hat{R}_{k\ov{\ell} i\ov{q}})=-\hat{g}^{\ov{\ell}k}\hat{g}^{\ov{q}i}\hat{g}_{j\ov{q}} \hat{\nabla}_{\ov{\ell}} \hat{T}^j_{ik}
=-\hat{g}^{\ov{\ell}k} \hat{\nabla}_{\ov{\ell}} \hat{T}^i_{ik},$$
so that
\begin{equation}\label{later}
\hat{g}^{\ov{\ell}k}(\hat{\nabla}_{\ov{\ell}} \hat{T}^i_{ik}+ \hat{R}_{i\ov{\ell}k\ov{q}}\hat{g}^{\ov{q}i}
- \hat{R}_{k\ov{\ell} i\ov{q}}\hat{g}^{\ov{q}i})=0,
\end{equation}
establishing \eqref{long}.

We now give the estimates on $(I), (II), (III)$.  The bounds on $(II)$ and $(III)$ follow immediately from the definitions of these quantities.  It remains to prove the bound on $(I)$.

In the special case when $g$ and $\hat{g}$ are K\"ahler, then it was shown by Aubin \cite{Au} and Yau \cite{Y} that $(I)\leq 0$.
To bound $(I)$ in general we follow Cherrier's generalization of this argument \cite{Ch}, as follows. Consider the inequality
$$K=\hat{g}^{\ov{\ell}i}g^{\ov{j}p}g^{\ov{q}k} B_{i\ov{j}k} \ov{B_{\ell\ov{p}q}}\geq 0,$$
where $$B_{i\ov{j}k}=\hat{\nabla}_i g_{k\ov{j}}- g_{i\ov{j}} \frac{\hat{\nabla}_k \tr{\hat{g}}g}{\tr{\hat{g}}g}+C_{i\ov{j}k},$$
and $C_{i\ov{j}k}$ will be specified below.
Calculate
\[\begin{split}
K&=\hat{g}^{\ov{\ell}i}g^{\ov{j}p}g^{\ov{q}k}\hat{\nabla}_i g_{k\ov{j}}\hat{\nabla}_{\ov{\ell}} g_{p\ov{q}}
+\frac{1}{\tr{\hat{g}}g}g^{\ov{q}k}\hat{\nabla}_k \tr{\hat{g}}g \hat{\nabla}_{\ov{q}} \tr{\hat{g}}g \\
&\ \ \ -2\mathrm{Re}\left(
\hat{g}^{\ov{\ell}i}g^{\ov{q}k}\hat{\nabla}_i g_{k\ov{\ell}} \frac{\hat{\nabla}_{\ov{q}} \tr{\hat{g}}g}{\tr{\hat{g}}g}\right)
-2\mathrm{Re}\left( \hat{g}^{\ov{j}i}g^{\ov{q}k}C_{i\ov{j}k} \frac{\hat{\nabla}_{\ov{q}} \tr{\hat{g}}g}{\tr{\hat{g}}g}\right)\\
&\ \ \ +2\mathrm{Re}\left( \hat{g}^{\ov{\ell}i}g^{\ov{j}p}g^{\ov{q}k} C_{i\ov{j}k} \hat{\nabla}_{\ov{\ell}} g_{p\ov{q}}\right)+\hat{g}^{\ov{\ell}i}g^{\ov{j}p}g^{\ov{q}k} C_{i\ov{j}k} \ov{C_{\ell\ov{p}q}}.
\end{split}\]
Using the identity
$$\hat{\nabla}_ig_{k\ov{\ell}}=\hat{\nabla}_kg_{i\ov{\ell}}
+  (T_0)^p_{ik} (g_0)_{p \ov{\ell}}
-\hat{T}^p_{ik}g_{p\ov{\ell}},$$
in the third term we get
\[\begin{split}
K&=\hat{g}^{\ov{\ell}i}g^{\ov{j}p}g^{\ov{q}k}\hat{\nabla}_i g_{k\ov{j}}\hat{\nabla}_{\ov{\ell}} g_{p\ov{q}}
-\frac{1}{\tr{\hat{g}}g}g^{\ov{q}k}\hat{\nabla}_k \tr{\hat{g}}g \hat{\nabla}_{\ov{q}} \tr{\hat{g}}g \\
&\ \ \ -2\mathrm{Re}\left(
\hat{g}^{\ov{\ell}i}g^{\ov{q}k}  (T_0)^p_{ik} (g_0)_{p \ov{\ell}}
 \frac{\hat{\nabla}_{\ov{q}} \tr{\hat{g}}g}{\tr{\hat{g}}g}\right)\\
&\ \ \ -2\mathrm{Re}\left( \hat{g}^{\ov{j}i}g^{\ov{q}k}\big[C_{i\ov{j}k}-\hat{T}^p_{ik}g_{p\ov{j}}\big] \frac{\hat{\nabla}_{\ov{q}} \tr{\hat{g}}g}{\tr{\hat{g}}g}\right)\\
&\ \ \ +2\mathrm{Re}\left( \hat{g}^{\ov{\ell}i}g^{\ov{j}p}g^{\ov{q}k} C_{i\ov{j}k} \hat{\nabla}_{\ov{\ell}} g_{p\ov{q}}\right)
+\hat{g}^{\ov{\ell}i}g^{\ov{j}p}g^{\ov{q}k} C_{i\ov{j}k} \ov{C_{\ell\ov{p}q}},
\end{split}\]
and comparing this expression with $(I)$ we get
\[\begin{split}
(I) &= \frac{1}{\tr{\hat{g}}{g}} \bigg( -K-2\mathrm{Re}\left(
\hat{g}^{\ov{\ell}i}g^{\ov{q}k}  (T_0)^p_{ik} (g_0)_{p \ov{\ell}}
 \frac{\hat{\nabla}_{\ov{q}} \tr{\hat{g}}g}{\tr{\hat{g}}g}\right)\\
&\ \ \ -2\mathrm{Re}\left( \hat{g}^{\ov{j}i}g^{\ov{q}k}\big[C_{i\ov{j}k}-\hat{T}^p_{ik}g_{p\ov{j}}\big] \frac{\hat{\nabla}_{\ov{q}} \tr{\hat{g}}g}{\tr{\hat{g}}g}\right)\\
&\ \ \ +2\mathrm{Re}\left( \hat{g}^{\ov{\ell}i}g^{\ov{j}p}g^{\ov{q}k} [C_{i\ov{j}k}-\hat{T}^r_{ik}g_{r\ov{j}}\big]  \hat{\nabla}_{\ov{\ell}} g_{p\ov{q}}\right)\\
&\ \ \ +\hat{g}^{\ov{\ell}i}g^{\ov{j}p}g^{\ov{q}k} C_{i\ov{j}k} \ov{C_{\ell\ov{p}q}}
-g^{\ov{j}i}\hat{g}^{\ov{\ell}k} \hat{T}^p_{ik}\ov{\hat{T}^q_{j\ell}}g_{p\ov{q}} \bigg),
\end{split}\]
and so the obvious choice to make is $$C_{i\ov{j}k}= \hat{T}^p_{ik}g_{p\ov{j}},$$
which makes three terms disappear, and gives us
\begin{equation}\label{later2}
(I)=\frac{1}{\tr{\hat{g}}{g}}\Bigg( - K-2\mathrm{Re}\left(
\hat{g}^{\ov{\ell}i}g^{\ov{q}k} (T_0)^p_{ik}(g_0)_{p\ov{\ell}} \frac{\hat{\nabla}_{\ov{q}} \tr{\hat{g}}g}{\tr{\hat{g}}g}\right) \Bigg).
\end{equation}
Hence
\begin{equation}\label{ineq} \nonumber
(I) \leq  \frac{2}{(\tr{\hat{g}}g)^2}\mathrm{Re}\left(
\hat{g}^{\ov{\ell}i}g^{\ov{q}k} (T_0)^p_{ki}(g_0)_{p\ov{\ell}} \hat{\nabla}_{\ov{q}} \tr{\hat{g}}g\right),
\end{equation}
as required.
\end{proof}

\section{Maximal Existence Time for the flow} \label{section:maximal}

In this section we give  proofs of Theorem \ref{theorem:maximal} and Theorem \ref{theorem:surfaces}.

\begin{proof}[Proof of Theorem \ref{theorem:maximal}]

As an aside, note that
$T$ can also be defined by
\[T = \sup \{ T_0 \ge 0 \ | \ \forall t \in [0,T_0], \exists \psi \in C^{\infty}(M) \textrm{ with } \alpha_t + \ddbar \psi >0 \},\]
for $\alpha_t$ given by (\ref{alphat}).

Fix $T_0<T$.  We will show there exists a solution of (\ref{crf}) on $[0,T_0)$.   Define reference metrics $\hat{\omega}_t$ for $t \in [0,T_0]$ by
$$\hat{\omega}_t := \alpha_t + \frac{t}{T_0} \ddbar f_{T_0} = \frac{T_0-t}{T_0} \omega_0 + \frac{t}{T_0} (\alpha_{T_0} + \ddbar f_{T_0}),$$
with $f_{T_0}$ a function satisfying $\alpha_{T_0} + \ddbar f_{T_0} >0$.  Note that  these Hermitian metrics vary smoothly on the compact interval $[0,T_0]$ and hence  we have estimates on $\hat{\omega}_t$ which are uniform for $t$ in $[0,T_0]$.
It is convenient to write $\hat{\omega}_t = \omega_0 + t\chi$ where $\chi$ is given by
$$\chi = \frac{1}{T_0} \ddbar f_{T_0} - \Ric(\omega_0).$$

Define a volume form $\Omega = \omega_0^n e^{\frac{f_{T_0}}{T_0}}$, which satisfies $\ddbar \log \Omega = \chi = \ddt{} \hat{\omega}_t$.  Now consider the parabolic complex Monge-Amp\`ere equation
\begin{equation} \label{ma}
\ddt{\varphi} = \log \frac{(\hat{\omega}_t + \ddbar \varphi)^n}{\Omega}, \quad \hat{\omega}_t + \ddbar \varphi >0, \quad \varphi|_{t=0}=0.
\end{equation}

If $\vp$ solves \eqref{ma} on some time interval, then taking $\ddbar$ of \eqref{ma} shows that $\omega=\hat{\omega}_t+\ddbar\vp$ solves \eqref{crf} on the same time interval. Conversely, if $\omega$ solves \eqref{crf} on an interval contained in $[0,T_0]$ then we have
$$\ddt{}(\omega-\hat{\omega}_t)=-\Ric(\omega)-\chi=\ddbar \left(\log\frac{\omega^n}{\omega_0^n}-\frac{f_{T_0}}{T_0}\right)=\ddbar \log\frac{\omega^n}{\Omega},$$
so if we choose $\vp$ to solve
$$\ddt{\vp}=\log\frac{\omega^n}{\Omega},\quad \vp|_{t=0}=0,$$
which is an ODE in $t$ for each fixed point on $M$,
then we have $\ddt{}(\omega-\hat{\omega}_t-\ddbar\vp)=0$
so that indeed $\omega=\hat{\omega}_t+\ddbar \vp$ and $\vp$ satisfies \eqref{ma}. Therefore, the two flows \eqref{crf} and \eqref{ma} are essentially equivalent.

We know by standard parabolic theory that there exists a unique maximal solution of (\ref{ma}) on some time interval $[0,\Tmax)$ with $ \Tmax >0$.  We may as well assume that   $\Tmax \le T_0$.  Assume for a contradiction that $\Tmax < T_0$.

We now prove uniform estimates for $\varphi$ solving $(\ref{ma})$ up to the maximal time:

\begin{lemma} \label{lemma:Testimates} There is a positive constant $C_0$, independent of $t \in [0,T_{\emph{max}})$,  such that
\begin{enumerate}
\item[(i)] $\displaystyle{\| \varphi (t) \|_{C^0} \le C_0}$.
\smallskip
\item[(ii)] $\displaystyle{ \| \dot{\varphi} (t) \|_{C^0} \le C_0}$.
\smallskip
\item[(iii)] $\displaystyle{C_0^{-1} \omega_0 \le \omega(t) \le C_0 \omega_0}$.
\smallskip
\item[(iv)] For each $k=0,1,2, \ldots$, there exist constants $C_k$ such that $$\displaystyle{ \| \varphi (t) \|_{C^k(\omega_0)} \le C_k}.$$
\end{enumerate}
\end{lemma}
\begin{proof}
The proofs of (i) and (ii) follow almost verbatim from the K\"ahler case \cite{TZ}, but we include brief arguments for the reader's convenience.
For (i), put $\psi = \varphi - At$ for a constant $A>0$.  Suppose that a maximum of $\psi$ occurs at a point with $t>0$.  Then at this point, by the maximum principle,
$$0 \le \ddt{} \psi = \log \frac{ (\hat{\omega}_t + \ddbar \psi)^n}{\Omega} -A <0,$$
if $A$ is chosen sufficiently large, a contradiction.  Here we are using the fact that $\hat{\omega}_t$ is a smooth family of metrics on $[0, \Tmax]$.  This gives an upper bound for $\psi$ and hence $\varphi$.  The lower bound is proved similarly.

For a lower bound for $\dot{\varphi}$, first note that
$$\ddt{} \dot{\varphi} = \Delta \dot{\varphi} + \tr{\omega}{\chi}.$$
Put $Q_0 = (T_0-t) \dot{\varphi} + \varphi + nt$ and compute
$$\left( \ddt{} - \Delta \right) Q_0 = (T_0-t) \tr{\omega}{\chi} + n - \Delta \varphi = \tr{\omega}{(\hat{\omega}_t + (T_0-t) \chi)} = \tr{\omega}{\hat{\omega}_{T_0}} >0.$$
Hence $Q_0$ is bounded below by the maximum principle and this gives a lower bound for $\dot{\varphi}$ (since we assume $\Tmax< T_0$.)

For the upper bound of $\dot{\varphi}$, define $Q_1 = t\dot{\varphi} - \varphi -nt$.  Then
$$\left( \ddt{} - \Delta \right) Q_1 = t \, \tr{\omega}{\chi} - n + \tr{\omega}{(\omega- \hat{\omega}_t)} = \tr{\omega}{(t\chi - \hat{\omega}_t)} = -\tr{\omega}{\omega_0} \le 0,$$
and an upper bound for $Q_1$ and hence $\dot{\varphi}$ follows from the maximum principle.

Note that by (ii) the volume form $\omega^n$ is uniformly equivalent to a fixed volume form $\omega_0^n$, say.  To prove (iii) then, it suffices to obtain a uniform upper bound of $\tr{g_0}{g}$.
 For this, we could apply the second order estimate  of Gill \cite{G}.  Instead, we give a different proof which uses a trick due to Phong-Sturm \cite{PS}, since we will use it again later in Sections \ref{section:divisor} and \ref{section:monge}.
Choose a constant $\tilde{C}$ so that $\varphi + \tilde{C} \ge 1$.  Following  \cite{PS}, we compute the evolution of
$$Q_2 = \log \tr{g_0}{g} - A \varphi + \frac{1}{\varphi+ \tilde{C}},$$
for $A$ a large constant to be determined.  We wish to show that at a point where $Q_2$ achieves a maximum, $\tr{g_0}{g}$ is uniformly bounded from above.  It will then follow from (i) that $\tr{g_0}{g}$ is bounded from above on $M$.

We apply Proposition \ref{prop:bigcalc} with $\hat{g}=g_0$ to obtain
\begin{align} \label{eqnfrombigcalc0}
\left( \ddt{} - \Delta \right) \log \tr{g_0}{g} & \le  \frac{2}{(\tr{g_0}{g})^2} \textrm{Re} \left( g^{\ov{\ell}k}  (T_0)^p_{kp}  \partial_{\ov{\ell}} \tr{g_0}{g} \right) + C \tr{g}{g_0},
\end{align}
assuming we are calculating at a point with $\tr{g_0}{g} \ge 1$.

To bound the first term on the right hand side of (\ref{eqnfrombigcalc0}), we note that at a maximum point of $Q_2$ we have $\partial_i Q_2=0$ and hence
\begin{equation}
\frac{1}{\tr{g_0}{g}} \partial_i \tr{g_0}{g} - A \partial_i \varphi - \frac{1}{(\varphi+\tilde{C})^2} \partial_i \varphi=0.
\end{equation}
Then at this maximum point for $Q_2$,
\begin{align} \nonumber \lefteqn{
\left| \frac{2}{(\tr{g_0}{g})^2} \textrm{Re} \left( g^{\ov{\ell}k}  (T_0)^p_{kp}  \partial_{\ov{\ell}} \tr{g_0}{g} \right) \right| } \\ \nonumber
& \le \left| \frac{2}{\tr{g_0}{g}} \textrm{Re} \left( \left(A+ \frac{1}{(\varphi+\tilde{C})^2}\right) g^{\ov{\ell} k} (T_0)^p_{kp} (\partial_{\ov{\ell}} \varphi )  \right) \right|  \\ \label{bd2Re}
& \le \frac{ | \partial \varphi|^2_{g}}{(\varphi +\tilde{C})^3} + C A^2 (\varphi+\tilde{C})^3 \frac{\tr{g}{g_0}}{(\tr{g_0}{g})^2},
\end{align}
for a uniform constant $C$.
But we may assume that at the maximum of $Q_2$ we have
$(\tr{g_0}{g})^2 \ge A^2 (\varphi + \tilde{C})^3,$
since otherwise we already have the required bound on $\tr{g_0}{g}$.

Thus at the maximum of $Q_2$, using (\ref{eqnfrombigcalc0}), (\ref{bd2Re}),
\begin{align} \nonumber
0 \le \left( \ddt{} - \Delta \right) Q_2  \le &\frac{ | \partial \varphi|^2_{g}}{(\varphi +\tilde{C})^3}  +C \tr{g}{g_0} - \left(A +\frac{1}{(\varphi +\tilde{C})^2}\right) \dot{\varphi}  \\ \nonumber
&  + \left( A + \frac{1}{(\varphi+\tilde{C})^2} \right)\tr{g}{(g - \hat{g}_t )} - \frac{2}{ (\varphi+\tilde{C})^3} | \partial \varphi|^2_g.
\end{align}
Since $\hat{g}_t \ge c_0 g_0$ with $t\in [0,\Tmax]$, for some uniform $c_0>0$, we may
choose $A$ sufficiently large so that $A \tr{g}{\hat{g}_t} \ge (C+1) \tr{g}{g_0}$.  Since $\dot{\varphi}$ is bounded from (ii), we obtain
$$\tr{g}{g_0} \le C'$$
at the maximum of $Q_2$, for a uniform constant $C'$.
Hence at the maximum of $Q_2$,
$$\tr{g_0}{g} \le \frac{1}{(n-1)!} (\tr{g}{g_0})^{n-1} \frac{\det g}{\det g_0} \le C'',$$
where we have applied (ii) again.   Hence $\tr{g_0}{g}$ is uniformly bounded from above on $M$, giving (iii).

Part (iv) follows from the higher order estimates of Gill \cite{G}.
\end{proof}

It is now straightforward to complete the proof of Theorem \ref{theorem:maximal}.
We have uniform estimates for the flow $\varphi(t)$ on $[0,\Tmax)$.  Taking limits we have a solution on $[0, \Tmax]$.  Applying the standard parabolic short time existence theory we obtain  a solution a little beyond $\Tmax$,  a contradiction.  Hence  there exists a unique solution of (\ref{ma}) on $[0,T_0)$.  Taking $\ddbar$ of (\ref{ma}) gives us a solution of (\ref{crf}) on $[0,T_0)$.  Since $T_0<T$ was chosen arbitrarily, we get a solution of (\ref{crf}) on $[0,T)$.

Uniqueness follows from uniqueness of solutions to (\ref{ma}).  Clearly the flow cannot extend beyond $T$.
\end{proof}

Next we give a proof of Theorem \ref{theorem:surfaces}.

\begin{proof}[Proof of Theorem \ref{theorem:surfaces}]
Note that in general if $\ddb\omega_0=0$ then $\ddb\omega=0$ for all later times.
If $n=2$ this means that the flow \eqref{crf} preserves the Gauduchon condition (recall that a Hermitian metric $\omega$ is \emph{Gauduchon} if $\partial \ov{\partial} \omega^{n-1}=0$).
The key result that we need is due to Buchdahl \cite{Bu2} and it
says that if $\eta$ is a $\de\db$-closed real $(1,1)$ form and $\omega_0$ a Gauduchon metric on a compact complex surface $M$
such that
$$\int_M\eta^2>0,\quad \int_M\eta\wedge\omega_0>0,\quad\int_D\eta>0,$$
for all irreducible effective divisors $D$ on $M$ with $D^2<0$, then there exists a smooth real function $f$ such that $\eta+\ddbar f>0$ is a Gauduchon metric. First of all observe that the following condition
\begin{equation}\label{cond}
\int_M\eta^2>0,\quad \int_M\eta\wedge\omega_0\geq 0,\quad\int_D\eta>0,
\end{equation}
for all $D$ as above is enough to guarantee the same conclusion. Indeed consider
the $(1,1)$ forms $\eta_t=\eta+t\omega_0$, and take $t>0$ large so that $\eta_t$ is a Gauduchon
metric. Then we have
$$\int_M\eta\wedge\eta_t = \int_M\eta^2+t\int_M\eta\wedge\omega_0>0,$$
and so we can apply Buchdahl's result using $\eta_t$ instead of $\omega_0.$
In particular, if \eqref{cond} holds then in fact we have the strict inequality $\int_M\eta\wedge\omega_0>0$.

We now apply this discussion to the $(1,1)$ forms $\alpha_t=\omega_0-t\Ric(\omega_0)$. As we have seen earlier,
the evolving metrics $\omega(t)$ are of the form $\omega(t)=\alpha_t+\ddbar \vp_t$, and so it follows that $T$ is the supremum of all $t\geq 0$ such that
\begin{equation}\label{cond2}
\int_M \alpha_t^2>0, \int_M \alpha_t\wedge\omega_0\geq 0,\int_D \alpha_t>0,
\end{equation}
for all $D$ as above. Furthermore, if \eqref{cond2} holds at some time $t$ then in fact
$\int_M \alpha_t \wedge\omega_0>0.$ Therefore $T$ is also the supremum of all $T_0\geq 0$ such that
\[\int_M \alpha_t^2>0,\int_D \alpha_t>0,\]
hold for all $t\in [0,T_0]$.
\end{proof}

\section{Estimates away from a divisor} \label{section:divisor}

In this section we give the proof of Theorem \ref{theorem:blowdown}.   Let $\omega=\omega(t)$ be the solution of (\ref{crf}) on the maximal time interval $[0,T)$.  We assume in this section that $T<\infty$.

In addition, we make the assumption that there exists a smooth function $f_{T}$ on $M$ such that
\begin{equation} \label{nn}
\hat{\omega}_{T} := \alpha_T + \ddbar f_{T} \ge 0,
\end{equation}
where we recall from (\ref{alphat}) that $\alpha_T = \omega_0 - T \Ric(\omega_0)$.
In the K\"ahler case (\ref{nn}) corresponds to the condition that the limiting K\"ahler class has a nonnegative representative.  Define reference metrics
$$\hat{\omega}_t := \frac{1}{T} \left( (T-t) \omega_0 + t \hat{\omega}_T\right), \ \textrm{for} \ t\in [0,T).$$
Then by the same argument as in the beginning of Section \ref{section:maximal}, we may write $\omega(t) = \hat{\omega}_t + \ddbar \varphi$ where
 $\varphi=\varphi(t)$ solves the parabolic complex Monge-Amp\`ere equation
\begin{equation} \label{ma2}
\ddt{\varphi} = \log \frac{(\hat{\omega}_t + \ddbar \varphi)^n}{\Omega}, \quad \hat{\omega}_t + \ddbar \varphi >0, \quad \varphi|_{t=0}=0,
\end{equation}
for the smooth volume form $\Omega = \omega_0^n e^{\frac{f_{T}}{T}}$.

We begin with a proposition, which is exactly analogous to a result for the K\"ahler-Ricci flow \cite{Ti, TZ} (see also the expositions in \cite{SW, SW4}).

\begin{proposition}  \label{prop:nn} With the assumptions above, there exists a constant  $C$ such that for all $t\in [0,T)$,
\begin{enumerate}
\item[(i)]  $\displaystyle{\| \varphi (t) \|_{C^0} \le C.}$
\smallskip
\item[(ii)] $\displaystyle{\dot{\varphi}(t) \le C}.$
\end{enumerate}
\end{proposition}
\begin{proof}
The proof is exactly the same as in the K\"ahler case.  Briefly:  the upper bound of $\varphi$ follows from the same argument as in Lemma \ref{lemma:Testimates}.  For the lower bound of $\varphi$ observe that $\hat{\omega}_t = \frac{(T-t)}{T} \omega_0 + \frac{t}{T} \hat{\omega}_{T} \ge \frac{(T -t)}{T} \omega_0$ and hence $\hat{\omega}_t^n \ge c_0 (T-t)^n \Omega$ for a uniform $c_0>0$.  The lower bound of $\varphi$ follows from applying the maximum principle to the quantity
$$Q = \varphi + n(T-t) (\log (T-t)-1) - (\log c_0 -1)t.$$
Indeed if $Q$ achieves its mimimum at some point $(x,t)$ with $t>0$ then at $(x,t)$ we have $\ddbar \vp\geq 0$ and
$$0\geq \ddt{}Q\geq \log\frac{\hat{\omega}_t^n}{\Omega}-n\log(T-t)-\log c_0+1\geq 1,$$
a contradiction. Hence $Q$ achieves its minimum at time $t=0$ which gives the lower bound for $\vp$.

The upper bound of $\dot{\varphi}$ follows from the same argument as in Lemma \ref{lemma:Testimates}.
\end{proof}

Next we prove a parabolic Schwarz lemma for volume forms for the flow (\ref{crf}).  It holds in the special case that $\hat{\omega}_{T}$ is the pull-back of a metric from another Hermitian manifold via a holomorphic map.  In the K\"ahler case this is due to Song-Tian \cite{ST}, and is a parabolic version of Yau's (volume) Schwarz lemma \cite{Y2}.

\begin{proposition} \label{prop:schwarz}  Let $\omega=\omega(t)$ solve (\ref{crf}) on $[0,T)$ with $T<\infty$. Suppose we have a holomorphic map $\pi :M \to N$ for $N$ a compact Hermitian manifold  of the same dimension $n$, equipped with a Hermitian metric $\omega_N$.  Suppose that $\hat{\omega}_{T} = \pi^* \omega_N$.
Then there exists a uniform constant $C>0$ such that on $M \times [0,T)$,
\begin{equation} \nonumber
\omega^n \ge \frac{1}{C}  (\pi^*\omega_N)^n.
\end{equation}
\end{proposition}
\begin{proof} Call
$$u=\frac{(\pi^*\omega_N)^n}{\omega^n}$$
the ratio of the volume forms. At any point where $u>0$ we
can calculate
$$\left(\frac{\de}{\de t}-\Delta\right)\log u=\tr{\omega}{\pi^*\Ric(\omega_N)}- \tr{\omega}{\Ric(\omega)}-g^{\ov{j}i} \frac{\de}{\de t}g_{i\ov{j}}=\tr{\omega}{\pi^*\Ric(\omega_N)}.$$
We apply the maximum principle to
$$Q = \log u - A \varphi - An(T-t) (\log (T-t)-1),$$
for a constant $A$ to be determined (cf. \cite[Lemma 7.3]{SW4}).
The maximum of $Q$ is achieved at a point where $u>0$, so we compute
\begin{align*}
\left( \ddt{} - \Delta\right) Q & = \tr{\omega}{\pi^*\Ric(\omega_N)} - A \dot{\varphi} + An \log(T-t) + A\tr{\omega}{(\omega-\hat{\omega}_t)} \\
& = - \tr{\omega}{ \left((A-1) \hat{\omega}_t  -\pi^*\Ric(\omega_N) \right)  } - A \log \frac{\omega^n}{\Omega (T-t)^n}\\
&\ \ \ \ \!  - \tr{\omega}{\hat{\omega}_t} + An.
\end{align*}
Choose $A$ sufficiently large so that for all $t\in [0,T]$,
$$(A-1) \hat{\omega}_t - \pi^*\Ric(\omega_N) = \frac{(A-1)}{T} \left( (T -t) \omega_0 + t \pi^* \omega_N \right) - \pi^*\Ric(\omega_N) \ge \pi^* \omega_N.$$
Note that by the arithmetic-geometric means inequality,
$$\tr{\omega}{\hat{\omega}_t} \ge \frac{(T-t)}{T} \tr{\omega}{\omega_0} \ge c \left( \frac{(T-t)^n \Omega}{\omega^n} \right)^{1/n},$$
for a uniform $c>0$.  Then
\begin{align*}
\left( \ddt{} - \Delta\right) Q&  \le - \tr{\omega}{\pi^* \omega_N} + A \log \frac{ (T-t)^n \Omega}{\omega^n} - c \left( \frac{(T-t)^n \Omega}{\omega^n} \right)^{1/n} + An \\
& \le - \tr{\omega}{\pi^* \omega_N} +C,
\end{align*}
using the fact that the map $x \mapsto A \log x - c x^{1/n}$ is uniformly bounded from above for $x>0$.  Thus at a maximum point of $Q$ we have
$\tr{\omega}{\pi^* \omega_N} \le C$, and applying the arithmetic-geometric means inequality again, $u$ is uniformly bounded from above at this point.  Since $\varphi$ is uniformly bounded by Proposition \ref{prop:nn}, this implies that $Q$ is bounded from above, and hence so is $u$.
\end{proof}

Now assume that we are in the situation of Theorem \ref{theorem:blowdown}.  The map $\pi: M \rightarrow N$ is a holomorphic map blowing down an exceptional divisor $E$ to a point $p\in N$.  More explicitly, a neighborhood of $E$ in $M$ can be identified with
$$\tilde{B} = \{ (z, \ell) \in B \times \mathbb{P}^{n-1} \ | \ z \in \ell \},$$
where $B$ is the open unit ball in $\mathbb{C}^n$ and elements $\ell$ in $\mathbb{P}^{n-1}$ are identified with lines through the origin in $\mathbb{C}^n$.  The map $\pi$ on $\tilde{B}$ is identified with the projection $(z, \ell) \mapsto z \in B$  and the exceptional divisor $E \subset M$ with the set $\pi^{-1}(0) \subset \tilde{B}$.   $\pi$ is a biholomorphism from $M \setminus E$ to $N \setminus \{p \}$.

By assumption, there exists a function $\psi=f_{T}$ with
$$\hat{\omega}_{T}:= \alpha_T + \ddbar f_{T} = \pi^*\omega_N.$$
Now  there exists a Hermitian metric $h$ on the fibers of the line bundle $[E]$ associated to the divisor $E$ with the property that for $\ve>0$ sufficiently small,
\begin{equation} \label{localhermitian}
\pi^* \omega_N - \ve R_h >0, \quad \textrm{with } R_h = - \ddbar \log h.
\end{equation}
For a proof of this statement, see \cite[p. 187]{GH}.  Although it is stated there in the K\"ahler case, the same proof carries over with $\omega_N$ Hermitian.
 Fix $s$ a holomorphic section of $[E]$ vanishing along $E$ to order 1.

We have a lemma:

\begin{lemma}\label{away}
With these hypotheses, there exists $A>0$ and $C$ such that
$$\emph{tr}_{g_0} \, g \le \frac{C}{|s|_h^{2A}}.$$
\end{lemma}
\begin{proof}   Define, as in Tsuji's work \cite{Ts}, $\tilde{\varphi} = \varphi - \ve_0 \log |s|_h^2$, which is uniformly bounded from below and goes to infinity on $E$. Here $\ve_0>0$ is a small constant that will be specified below.  Choose a constant $C_0$ so that $\tilde{\varphi} +C_0 \ge 1$.  Following Phong-Sturm \cite{PS} (and as  in Lemma \ref{lemma:Testimates} above), we compute the evolution of
$$Q = \log \tr{g_0}{g} - A \tilde{\varphi} + \frac{1}{\tilde{\varphi}+C_0},$$
for $A$ to be determined (assume at least $A \ve_0 >1)$.  Note that the quantity $1/(\tilde{\varphi}+C_0)$ is bounded (in fact it lies between 0 and 1).  Moreover, $Q$ tends to negative infinity on $E$ and hence for each fixed time $t$, the quantity $Q(x,t)$ achieves a maximum at some point in $M \setminus E$.

From Proposition \ref{prop:bigcalc} we have
\begin{align} \label{eqnfrombigcalc}
\left( \ddt{} - \Delta \right) \log \tr{g_0}{g} & \le  \frac{2}{(\tr{g_0}{g})^2} \textrm{Re} \left( g^{\ov{\ell}k}  (T_0)^p_{kp}  \partial_{\ov{\ell}} \tr{g_0}{g} \right) + C \tr{g}{g_0},
\end{align}
assuming we are calculating at a point with $\tr{g_0}{g} \ge 1$.
To bound the first term on the right hand side, we note that at a maximum point of $Q$ we have $\partial_i Q=0$ and hence
\begin{equation} \nonumber
\frac{1}{\tr{g_0}{g}} \partial_i \tr{g_0}{g} - A \partial_i \tilde{\varphi} - \frac{1}{(\tilde{\varphi}+C_0)^2} \partial_i \tilde{\varphi}=0.
\end{equation}
Thus at this maximum point for $Q$,
\begin{align} \nonumber \lefteqn{
\left| \frac{2}{(\tr{g_0}{g})^2} \textrm{Re} \left( g^{\ov{\ell}k}  (T_0)^p_{kp}  \partial_{\ov{\ell}} \tr{g_0}{g} \right) \right| } \\ \nonumber
& \le \left| \frac{2}{\tr{g_0}{g}} \textrm{Re} \left( \left(A+ \frac{1}{(\tilde{\varphi}+C_0)^2}\right) g^{\ov{\ell} k} (T_0)^p_{kp} (\partial_{\ov{\ell}} \tilde{\varphi} )  \right) \right|  \\ \nonumber
& \le \frac{ | \partial \tilde{\varphi}|^2_{g}}{(\tilde{\varphi} +C_0)^3} + C A^2 (\tilde{\varphi}+C_0)^3 \frac{\tr{g}{g_0}}{(\tr{g_0}{g})^2},
\end{align}
for a uniform constant $C$.
If at the maximum of $Q$ we have $(\tr{g_0}{g})^2\leq A^2(\tilde{\varphi}+C_0)^3$ then at
the same point we have
$$ (\tr{g_0}{g}) |s|_h^{2A\ve_0}\leq A(\varphi-\ve_0\log|s|^2_h+C_0)^{\frac{3}{2}}|s|_h^{2A\ve_0}\leq C_A,$$
for a constant $C_A$ depending on $A$, and so
$$Q = \log(( \tr{g_0}{g}) |s|_h^{2A\ve_0}) - A \varphi + \frac{1}{\tilde{\varphi}+C_0} \le C'_A,$$
and we are done.  If on the other hand at the maximum of $Q$ we have $A^2(\tilde{\varphi}+C_0)^3 \le (\tr{g_0}{g})^2$ then
\begin{align} \nonumber
\left| \frac{2}{(\tr{g_0}{g})^2} \textrm{Re} \left( g^{\ov{\ell}k}  (T_0)^p_{kp}  \nabla_{\ov{\ell}} \tr{g_0}{g} \right) \right|
& \le \frac{ | \partial \tilde{\varphi}|^2_{g}}{(\tilde{\varphi} +C_0)^3}  + C \tr{g}{g_0}.
\end{align}
Now compute at the maximum of $Q$, using (\ref{eqnfrombigcalc})
\begin{align} \nonumber
0 \le \left( \ddt{} - \Delta \right) Q & \le \frac{ | \partial \tilde{\varphi}|^2_{g}}{(\tilde{\varphi} +C_0)^3}  +C \tr{g}{g_0} - \left(A +\frac{1}{(\tilde{\varphi} +C_0)^2}\right) \dot{\varphi}  \\ \nonumber
& \ \ \  + \left( A + \frac{1}{(\tilde{\varphi}+C_0)^2} \right)\tr{\omega}{(\omega - (\hat{\omega}_t - \ve_0 R_h))}\\
&\ \ \  - \frac{2}{ (\tilde{\varphi}+C_0)^3} | \partial \tilde{\varphi}|^2_g.
\end{align}
Since we clearly have $\hat{\omega}_t\geq c \hat{\omega}_T$ for some constant $c>0$,
we can use (\ref{localhermitian}) to get that $\hat{\omega}_t - \ve_0 R_h \ge c_0 \omega_0$ for some uniform $c_0>0$,
provided we choose $\ve_0$ sufficiently small.  Hence we may
choose $A$ sufficiently large so that
$$\tr{g}{g_0} \le C \log \frac{\Omega}{\omega^n} +C.$$
Hence at the maximum of $Q$,
$$\tr{g_0}{g} \le \frac{1}{(n-1)!} (\tr{g}{g_0})^{n-1} \frac{\det g}{\det g_0} \le
C\frac{\omega^n}{\Omega} \left(\log \frac{\Omega}{\omega^n}\right)^{n-1}+C\leq C',$$
because we know that $\frac{\omega^n}{\Omega}\leq C$ and $x\mapsto x|\log x|^{n-1}$
is bounded above for $x$ close to zero. This implies that $Q$ is bounded from above at its maximum, and completes the proof of the lemma.
\end{proof}

It is now straightforward to complete the proof of the theorem.

\begin{proof}[Proof of Theorem \ref{theorem:blowdown}]
We apply Proposition \ref{prop:schwarz} and Lemma \ref{away} to see that for any compact set $K\subset M \setminus E$, there exists a constant $C_K>0$ such that
$$C_K^{-1} \omega_0 \le \omega(t) \le C_K\omega_0 \quad \textrm{on } K \times [0,T).$$
Applying the higher order estimates of Gill \cite{G}, which are local, we obtain uniform $C^{\infty}$ estimates for $\omega(t)$ on compact subsets of $M \setminus E$.  In particular,
for every compact set $K$ there exists a constant $C'_K$ such that
$$\ddt{} \omega = - \Ric(\omega) \le C'_K \omega,$$
which implies that $e^{-C'_Kt} \omega(t)$ is decreasing in $t$ as well as being bounded from below.  This implies that a limit for $\omega(t)$ exists as $t \rightarrow T$, and since we have uniform estimates away from $E$, we see that $\omega(t)$ converges in $C^{\infty}$ on compact subsets to a smooth Hermitian metric $\omega_{T}$ on $M \setminus E$.
\end{proof}

\section{The Chern-Ricci flow on complex surfaces} \label{section:surfaces}

In this section we give a proof of Theorem \ref{theorem:surfaces2}, and we pose some conjectures on the behavior of the Chern-Ricci flow on surfaces, and its relation to the `minimal model program for complex surfaces'.

First recall that the Kodaira dimension of a compact complex manifold $M$ of dimension $n$ is given by
$$\kappa(M)=\limsup_{\ell\to +\infty}\frac{\log \dim H^0(M,\ell K_M)}{\log \ell}\in \{-\infty,0,1,\dots,n\}.$$

\begin{proof}[Proof of Theorem \ref{theorem:surfaces2}]
(a) We need to show that if $T=\infty$ then $M$ is minimal. If $M$ is a non-minimal compact complex surface then we must have $T<\infty$, because if $D$ is any $(-1)$-curve in $M$ we have that $D\cdot K_M<0$ and so the volume of $D$,
$$\int_D\omega(t)=\int_D\omega_0 +2\pi t D\cdot K_M,$$
becomes zero in finite time.

(b) If the volume goes to zero at time $T<\infty$, we have that $\int_M \alpha_T^2=0$,
where we recall from (\ref{alphat}) that $\alpha_T$ is the $\partial \ov{\partial}$-closed $(1,1)$ form $\alpha_T=\omega_0-T \Ric(\omega_0)$.
 We claim that in this case the Kodaira dimension $\kappa(M)$ is negative, which by the Kodaira-Enriques classification \cite{bhpv, Be} implies that $M$ is either birational to a ruled surface or of class VII. Indeed, if $\kappa(M)\geq 0$ then
some power $\ell K_M$, $\ell\geq 1$ of the canonical bundle would be effective. Let $E$ be an effective divisor in $|\ell K_M|$, then $E$ must be nonempty since otherwise $\ell K_M$ would be trivial, and so $c_1^{\mathrm{BC}}(M)=0$ and by Theorem \ref{theorem:gill} we would have $T=\infty$.
We thus conclude that $E$ is nonempty, and let
$s$ be a section of $\ell K_M$ defining $E$ and $h$ a smooth metric on the fibers of $\ell K_M$.
Then by the Poincar\'e-Lelong formula the curvature $\eta$ of $h$ is a smooth closed real $(1,1)$ form such
that
$$2\pi [E]=\eta+\ddbar \log |s|^2_h$$
holds as currents on $M$ (see also \cite{Ga2}). In particular for any $\de\db$-closed smooth $(1,1)$ form $\gamma$ we have
$$\int_M \gamma\wedge\eta=2\pi\int_E\gamma.$$
Moreover if $\gamma$ is a Gauduchon metric we have
$$\int_M \gamma\wedge c_1^{\mathrm{BC}}(M)=-\frac{2\pi}{\ell}\int_E\gamma<0,$$
because $-\frac{1}{\ell}\eta$ represents $c_1^{\mathrm{BC}}(M)$.
Taking $\gamma=\omega(t)$ and letting $t$ approach $T$ we get
$$\int_M \alpha_T\wedge c_1^{\mathrm{BC}}(M)\leq 0.$$
We also have
\[\begin{split}
0=\int_M \alpha_T^2&=-T\int_M \alpha_T\wedge c_1^{\mathrm{BC}}(M) +\int_M \alpha_T\wedge\omega_0\\
&\geq \int_M \alpha_T\wedge\omega_0\geq 0,
\end{split}\]
which implies that $\int_M \alpha_T\wedge\omega_0=0$. Applying \cite[Lemma 4]{Bu1} we see that
$\alpha_T=\ddbar f$ for some smooth function $f$, which implies that $\omega_0$ is K\"ahler
and that $M$ is Fano, contradicting the assumption that the Kodaira dimension of $M$ is nonnegative.
To see that $M$ cannot be an Inoue surface, we apply the observation (below) that on an Inoue surface, the Chern-Ricci flow exists for all time.

(c) Assume now that $T<\infty$ and that the volume does not collapse at time $T$, so that $\int_M \alpha_T^2>0$.
We know from Theorem \ref{theorem:maximal} that there is no smooth function $f$ such that $\alpha_T+\ddbar f>0$ (otherwise we could continue the flow past $T$). On the other hand for $\ve>0$
$$\alpha_T+\ve\omega_0=(1+\ve)\omega_0-T\Ric(\omega_0)=(1+\ve)\left( \omega_0-\frac{T}{1+\ve}\Ric(\omega_0)\right),$$
and since $\frac{T}{1+\ve}<T$ we have $\omega_0-\frac{T}{1+\ve}\Ric(\omega_0)+\ddbar f>0$ for some function $f$.
Therefore
$$\int_M \omega_0\wedge(\alpha_T+\ve\omega_0)>0,$$
and letting $\ve\to 0$ we get $\int_M\alpha_T\wedge\omega_0\geq 0$, therefore
$$\int_M \alpha_T\wedge (\alpha_T+\ve\omega_0)=\int_M \alpha_T^2+\ve\int_M\alpha_T\wedge\omega_0>0.$$
We now apply the main theorem of \cite{Bu2}, and we see that there is an irreducible effective divisor $D\subset M$
with $D^2<0$ such that $\int_D\alpha_T=0$. Furthermore
\begin{equation}\label{dotneg}
2\pi K_M\cdot D=-\int_D\Ric(\omega_0)=\frac{1}{T}\int_D (\alpha_T-\omega_0)
=-\frac{1}{T}\int_D\omega_0<0,
\end{equation}
and so by the adjunction formula $D$ is a smooth $(-1)$-curve.

Assume now that $M$ is minimal and consider first the case when $M$ is K\"ahler. If $\kappa(M)\geq 0$ then $K_M$ is nef thanks to \cite[Corollary III.2.4]{bhpv}, while if $\kappa(M)=-\infty$ then by the Kodaira-Enriques classification $M$ is either $\mathbb{CP}^2$ or ruled. If $K_M$ is nef then $\int_M c_1^{\mathrm{BC}}(M)^2\geq 0$ and if $\kappa(M)\geq 0$ then some power of $K_M$ is effective, so the argument above implies that $\int_M\omega_0 \wedge c_1^{\mathrm{BC}}(M)<0$. Thus the volume along the flow is
$$V_t=\int_M \omega(t)^2=\int_M\omega_0^2-2t\int_M \omega_0\wedge c_1^{\mathrm{BC}}(M)+t^2\int_M c_1^{\mathrm{BC}}(M)^2,$$
which is always positive, so by Theorem \ref{theorem:surfaces} we have $T=\infty$.
On the other hand if $M$ is $\mathbb{CP}^2$ or ruled then we must have $T<\infty$ and since case (c) is excluded we must be in case (b). Indeed, if $M$ is $\mathbb{CP}^2$ and $E$ is a line inside it
then $E\cdot K_M<0$ and its volume
$$\int_E \omega(t)=\int_E\omega_0+2\pi t E\cdot K_M$$
goes to zero in finite time. The other case is when $M$ is ruled and $E\cong\mathbb{CP}^1$ is a fiber of the
ruling then $E\cdot E=0$ and by the genus formula $E\cdot K_M=-2$, which again implies that the volume of $E$
goes to zero in finite time.

If on the other hand $M$ is not K\"ahler, thanks to the Kodaira-Enriques classification \cite{bhpv, Be} we know that minimal non-K\"ahler compact complex surfaces fall into the following classes:
\begin{enumerate}
\item Primary and secondary Kodaira surfaces,
\item Surfaces of class VII with $b_2(M)=0,$
\item Minimal surfaces of class VII with $b_2(M)>0,$
\item Minimal properly elliptic surfaces,
\end{enumerate}
where a surface of class VII is by definition a compact complex surface with $b_1(M)=1$ and $\kappa(M)=-\infty$, while a properly elliptic surface is an elliptic surface with $\kappa(M)=1$.
The surfaces in $(1)$, $(2)$ and $(4)$ are completely classified, and while there are many examples of surfaces
in $(3)$, a complete classification is still lacking (see e.g. \cite{bhpv, DOT, LYZ, N, T0, T1, T3}). We treat
each case separately.

In case $(1)$ the manifold $M$ has torsion canonical bundle (i.e. some power $\ell K_M$, $\ell\geq 1$ is holomorphically trivial). In particular these manifolds have $c_1^{\mathrm{BC}}(M)=0$, and Theorem \ref{theorem:gill} says that the Chern-Ricci flow starting from any initial Hermitian metric $\omega_0$ has a long time solution $\omega(t)$ (so we are in case (a)) which as $t$ goes to infinity converges smoothly to the unique Hermitian metric of the form $\omega_\infty=\omega_0+\ddbar\vp_\infty$ with $\Ric(\omega_\infty)=0$.

In case $(2)$, the manifold $M$ is either an Inoue surface or a Hopf surface \cite{LYZ, T0}.  Suppose first that $M$ is an Inoue surface.  Then  $M$ does not have any curves and by \cite[Remark 4.2]{T2} any Gauduchon metric $\omega_0$  satisfies
$\int_M\omega_0\wedge c_1^{\mathrm{BC}}(M)<0$. In particular the volume of $M$ along the flow is
$$V_t=\int_M \omega(t)^2=\int_M\omega_0^2-2t\int_M \omega_0\wedge c_1^{\mathrm{BC}}(M).$$
Since $V_t$ is always positive, Theorem \ref{theorem:surfaces} implies that the Chern-Ricci flow exists for
all positive time, so we are in case (a).

If $M$ is a Hopf surface, it follows from the arguments in \cite[Remark 4.3]{T2} that any Gauduchon metric $\omega_0$ on them satisfies
$\int_M\omega_0\wedge c_1^{\mathrm{BC}}(M)>0$. Indeed as we have seen this holds whenever $M$ has a plurianticanonical divisor, and
\cite[Remark 4.3]{T2} shows that every primary Hopf surface has an anticanonical divisor. But every Hopf surface is either primary or secondary (i.e. a finite unramified quotient $\ti{M}\to M$ of a primary one) and an anticanonical divisor on $\ti{M}$ gives a plurianticanonical divisor on $M$.
In particular the volume of $M$ along the flow is
$$V_t=\int_M \omega(t)^2=\int_M\omega_0^2-2t\int_M \omega_0\wedge c_1^{\mathrm{BC}}(M),$$
which goes to zero in finite time, and so the Chern-Ricci flow exists for finite time. In fact, since every curve on $M$ is homologous to zero, the flow exists precisely as long as the volume stays positive and then it collapses, so
we are in case (b).
We will investigate the behavior of the flow on a family of Hopf manifolds in Section \ref{section:hopf}.

In case $(3)$, if we call $b_2(M)=n>0$, we have $\int_M c_1^2(M)=-n$ (see e.g. \cite[p.494]{T1}). It follows
that the Chern-Ricci flow starting from any initial Gauduchon metric $\omega_0$ exists only for finite time, because
the volume of $M$ along the flow is
$$V_t=\int_M \omega(t)^2=\int_M\omega_0^2-2t\int_M \omega_0\wedge c_1^{\mathrm{BC}}(M)
-4\pi^2 nt^2,$$
which goes to zero in finite time. Furthermore, since $M$ is minimal and using again Theorem \ref{theorem:surfaces},
we see that we are in case (b).
Note that carrying out a space-time rescaling of
the flow to have constant volume will still produce a solution that exists only for a finite time (cf. the discussion in \cite{StT3}).

In case $(4)$ we have $\int_M c_1^2(M)=0$ and by definition some power of the
canonical bundle $\ell K_M$, $\ell\geq 1$, is effective. Arguing as before, this implies that
$$\int_M \omega_0\wedge c_1^{\mathrm{BC}}(M)<0.$$
Therefore the volume along the Chern-Ricci flow remains positive for all time,
and since $M$ is minimal the flow has a long time solution and we are in case (a).
\end{proof}

Furthermore, arguing like in \cite[Proposition 8.4]{SW4}, one can show that in case (b) the volume goes to zero quadratically only on a Fano manifold in a positive multiple of the anticanonical class, otherwise it goes to zero linearly.

We finish this section with some further discussions and conjectures.
We begin by considering what happens in case (c), along the lines of \cite[Theorem 8.3]{SW4}.
First of all note that $\alpha_T$ is a $\de\db$-closed real $(1,1)$ form with $\int_M\alpha_T^2>0$. It follows from \cite[Lemma 4]{Bu1} that if $\psi$ is another $\de\db$-closed real $(1,1)$ form with $\int_M \psi\wedge\alpha_T=0$
then $\int_M\psi^2\leq 0$ with equality if and only if $\psi=\ddbar f$ for some function $f$.
If now $D_1, D_2$ are irreducible distinct $(-1)$-curves (so $D_1\cdot D_2 \geq 0$) with $\int_{D_1}\alpha_T=\int_{D_2}\alpha_T=0$, then as before we can express the divisor $D_1+D_2$ as
$$2\pi[D_1+D_2]=\eta+\ddbar\log|s|^2_h,$$
in the sense of currents, where $\eta$ is a smooth closed form that represents $2\pi c_1(D_1+D_2)$.
Therefore, since $\de\db\alpha_T=0$,
$$0=\int_{D_1}\alpha_T+\int_{D_2}\alpha_T=\frac{1}{2\pi}\int_M\eta\wedge\alpha_T,$$
and so $4\pi^2(D_1+D_2)^2=\int_M\eta^2\leq 0$, with equality implying that $\eta=\ddbar f$.
But this would give
$$0=-\int_M \eta\wedge c_1^{\mathrm{BC}}(M)=2\pi K_M\cdot (D_1+D_2)<0,$$
a contradiction.  Thus  we conclude that $(D_1+D_2)^2<0$,
which implies that $D_1\cdot D_2=0$ and $D_1,D_2$ are disjoint. The set of all these $(-1)$-curves is finite, $D_1, \ldots, D_k$ say, because they give linearly independent classes in homology. Contracting all of them
we get a contraction map $\pi:M\to N$, where $N$ is a compact complex surface which is K\"ahler if and only if $M$ is.

In light of the behavior of the K\"ahler-Ricci flow on surfaces, it is natural to ask whether the Chern-Ricci flow contracts, in the sense of \cite{SW2}, the $(-1)$-curves $D_1, \ldots, D_k$  to points $p_1, \ldots, p_k$ on $N$.  First, do the metrics $\omega(t)$ converge smoothly on compact subsets of $M \setminus \bigcup_i D_i$ to a smooth K\"ahler metric $\omega(T)$ on $M \setminus \bigcup_i D_i$, as in Theorem \ref{theorem:blowdown}?  This would hold if we can find $\beta$, a $\de\db$-closed real $(1,1)$ form on $N$, and $f$ a smooth function on $M$
such that $\pi^*\beta=\alpha_T+\ddbar f$.

Furthermore, does the
 family $(M, \omega(t))$ converge in the sense of Gromov-Hausdorff to a limiting compact metric space $(N, d)$ as $t \rightarrow T^-$?  Can we produce a solution of the Chern-Ricci flow $\tilde{\omega}(t)$ on $N$ for $t \in [T, T']$ (with $T'>T$)  such that $\tilde{\omega}(T)$ on $N \setminus \{ p_1, \ldots, p_k \}$ can be identified with $\omega(T)$ via the blow-down map?  Does the family $(N, \tilde{\omega}(t))$ converge in the Gromov-Hausdorff sense to $(N,d)$ as $t \rightarrow T^+$?

If this can be carried out, one could continue this process a finite number of times to obtain a solution of the Chern-Ricci flow `with canonical surgical contractions' \cite{SW2} all the way to the minimal model of $M$.

Finally, what is the long time behavior of the Chern-Ricci flow on a minimal model $M$?  In the case when $M$ has Kodaira dimension zero,
so that it has torsion canonical bundle (and is either a Calabi-Yau surface or a Kodaira surface)  the flow
always converges to a Chern-Ricci flat metric (which need not be K\"ahler, even if $M$ is Calabi-Yau) by Gill's Theorem \ref{theorem:gill}.   Another case (of course there are many more) is when $c_1(M)<0$.  We discuss this case, for any dimension, in the next section.

\section{Convergence when $c_1(M)<0$}\label{section:c1neg}
In this section we assume that $M$ is a compact K\"ahler manifold with $c_1(M)<0$ and we give the proof of Theorem \ref{theorem:c1n}.

Start by fixing a smooth volume form $\Omega$ with $\Ric(\Omega)<0$, which is possible
because $c_1(M)<0$, and note that $\Ric(\Omega)$ also represents $c_1^{\mathrm{BC}}(M)$. Therefore, for all $t>0$ the $(1,1)$ form $\omega_0-t\Ric(\Omega)$ is positive, and by Theorem \ref{theorem:maximal} the Chern-Ricci flow \eqref{crf} exists for all time.  Call $\ti{\omega}(s)$ its solution.

We consider now the rescaled metrics $\omega=\frac{\ti{\omega}}{s+1}$ and a new time parameter $t=\log(s+1)$, so that
the new metrics solve
\begin{equation} \label{crf2}
\frac{\de}{\de t} \omega = - \Ric(\omega)-\omega, \quad \omega|_{t=0}=\omega_0,
\end{equation}
for all positive $t$.
First of all we show that \eqref{crf2} is equivalent to a parabolic complex Monge-Amp\`ere equation.
To see this, call $\hat{\omega}=-\Ric(\Omega)+e^{-t}(\Ric(\Omega)+\omega_0),$ and note that they are Hermitian metrics that satisfy
\begin{equation}
\ddt{} \hat{\omega} = - \Ric(\Omega)-\hat{\omega}, \quad \hat{\omega}|_{t=0}=\omega_0,
\end{equation}
and $\hat{\omega}$ converges smoothly to $-\Ric(\Omega)$ as $t$ goes to infinity.
It follows that
$$\ddt{} (\omega-\hat{\omega})=-(\omega-\hat{\omega})+\ddbar\log\frac{\omega^n}{\Omega}.$$
Consider now the solution $\vp$ of the equation
\begin{equation} \label{pma}
\ddt{} \vp = \log\frac{\omega^n}{\Omega}-\vp,\quad \vp|_{t=0}=0,
\end{equation}
which exists for all positive time as can be seen by regarding it as an ODE in $t$ for each fixed point on $M$. We have that
$$\ddt{} \left(e^t(\omega-\hat{\omega}-\ddbar\vp)\right)=0, \quad (\omega-\hat{\omega}-\ddbar\vp)|_{t=0}=0,$$
which implies that $\omega=\hat{\omega}+\ddbar\vp$ holds for $t\geq 0$. Then Theorem \ref{theorem:c1n} follows directly from:

\begin{theorem}\label{conv}
As $t\to \infty$ we have that $\vp\to\vp_\infty$ smoothly, and $\omega_\infty:=-\Ric(\Omega)+\ddbar\vp_\infty$
equals the unique K\"ahler-Einstein metric $\omega_{\mathrm{KE}}$.
\end{theorem}

\begin{proof}
First, we derive uniform estimates for $\vp$ independent of $t$.  The estimates for $\| \varphi\|_{C^0}$ and $\| \dot{\varphi} \|_{C^0}$ follow from the same arguments as in \cite{Ca,  TZ, Ts}. Indeed, a simple maximum principle argument shows that $|\vp|\leq C$ independent of $t$.  Compute
\[\begin{split}
\left(\ddt{}-\Delta\right)\dot{\vp}&=-\dot{\vp}-\tr{\omega}{(\Ric(\Omega)+\hat{\omega})}\\
&=-\dot{\vp}-n+\Delta\vp-\tr{\omega}{\Ric(\Omega)},
\end{split}\]
and so
$$\left(\ddt{}-\Delta\right)(\vp+\dot{\vp})=-n-\tr{\omega}{\Ric(\Omega)}.$$
At the minimum of $\vp+\dot{\vp}$, assuming it occurs for $t>0$, we have
$-\tr{\omega}{\Ric(\Omega)}\leq n,$ and since $\Ric(\Omega)<0$ the arithmetic-geometric means inequality
gives us $\vp+\dot{\vp}=\log\frac{\omega^n}{\Omega}\geq -C$ at this point and hence everywhere. Since $|\vp|\leq C$, we get $\dot{\vp}\geq -C$. But we also have that $\Ric(\Omega)+\hat{\omega}=e^{-t}(\Ric(\Omega)+\omega_0),$ and so
$$\left(\ddt{}-\Delta\right)(\vp+\dot{\vp}+nt-e^t\dot{\vp})=\tr{\omega}{\omega_0}>0,$$
which implies by the maximum principle that, for $t\ge 1$, $\dot{\vp}\leq Cte^{-t}\leq Ce^{-t/2}$.  This estimate is the same as the one in \cite{TZ}.

We feed this into the second order estimate as before. We have
\begin{align}  \nonumber
\left( \ddt{} - \Delta \right) \log \tr{g_0}{g} & \le  \frac{2e^{-t}}{(\tr{g_0}{g})^2} \textrm{Re} \left( g^{\ov{\ell}k}  (T_0)^p_{kp}  \partial_{\ov{\ell}} \tr{g_0}{g} \right) + C \tr{g}{g_0}-1,
\end{align}
assuming we are calculating at a point with $\tr{g_0}{g} \ge 1$. Indeed, this follows from the calculations of Proposition \ref{prop:bigcalc} with $\hat{g}=g_0$, with the minor change that now $g(t)$ evolves by the normalized Chern-Ricci flow. In particular, we now have $d\omega=e^{-t}d\omega_0$ instead of $d\omega=d\omega_0$.

Arguing  as in the proof of Lemma \ref{lemma:Testimates} we get that $\omega$ is uniformly equivalent to $\omega_0$ independent of $t$. Uniform higher order estimates are then provided by Gill's paper \cite{G}.
In particular, it follows that
$$\left(\ddt{}-\Delta\right)(e^t\dot{\vp})=-\tr{\omega}{(\Ric(\Omega)+\omega_0)}\geq -C,$$
and so the maximum principle implies that $\dot{\vp}\geq -C(1+t)e^{-t}\geq -Ce^{-t/2}.$
This implies that as $t$ approaches infinity $\dot{\vp}$ converges uniformly to zero exponentially fast, which implies that $\vp$ converges uniformly exponentially fast to a continuous limit function $\vp_\infty$. Since we have uniform higher order estimates for $\vp$, it follows that $\vp_\infty$ is actually smooth and the convergence of $\vp$ to $\vp_\infty$ is in the smooth topology. Therefore we can pass to the limit in \eqref{pma} and see that the limiting metric $\omega_\infty=-\Ric(\Omega)+\ddbar\vp_\infty$ satisfies
$$\log\frac{\omega_\infty^n}{\Omega}=\vp_\infty,$$
and taking $\ddbar$ of this, we get
$$\Ric(\omega_\infty)=-\omega_\infty,$$
so that $\omega_\infty$ is the unique K\"ahler-Einstein metric on $M$.
\end{proof}

\section{Hopf Manifolds} \label{section:hopf}

In this section we study the Chern-Ricci flow on some Hopf manifolds.

As in the Introduction, for $\alpha = (\alpha_1, \ldots, \alpha_n) \in \mathbb{C}^n \setminus \{ 0 \}$ with  $|\alpha_1|=\dots=|\alpha_n| \neq 1$, let $M_{\alpha}$ be the Hopf manifold
 $M_{\alpha}=(\mathbb{C}^n\setminus \{0\})/\sim$, where
$$(z_1, \dots,z_n) \sim \left(\alpha_1 z_1, \dots, \alpha_n z_n\right).$$
We consider the metric
$$\omega_H = \frac{\delta_{ij}}{r^2} \mn dz_i \wedge d\ov{z}_j,$$
where $r^2=\sum_{j=1}^n |z_j|^2$. If $n=2$, $\omega_H$ is $\ddb$-closed, but this is false if $n>2$.
We now show that $\omega(t) = \omega_H - t \Ric(\omega_H)$ gives an explicit solution of the Chern-Ricci flow on $M_{\alpha}$.

\begin{proof}[Proof of Proposition \ref{prop:hopf1}]
Observe that  $\det(\omega_H)=r^{-2n}$ and
$$\Ric(\omega_H) = n \ddbar \log r^2 = \frac{n}{r^2} \left(\delta_{ij} - \frac{\ov{z}_i z_j}{r^2}\right)\mn dz_i \wedge d\ov{z}_j \ge 0.$$
For $t<\frac{1}{n}$ we have the Hermitian metrics
$$\omega(t)=\omega_H-t\Ric(\omega_H)=\frac{1}{r^2}\left((1-nt)\delta_{ij}+nt\frac{\ov{z}_i z_j}{r^2}\right)
\mn dz_i \wedge d\ov{z}_j.$$
To compute the determinant of $\omega(t)$, note that the matrix $nt\frac{\ov{z}_i z_j}{r^2}$ has eigenvalue
$nt$ with multiplicity $1$ and all the other eigenvalues are zero, while the matrix $(1-nt)\delta_{ij}$
has eigenvalue $1-nt$ with multiplicity $n$ and is diagonal in every coordinate system.
Choosing a coordinate system that makes $\frac{\ov{z}_i z_j}{r^2}$ diagonal we see that the eigenvalues
of $r^2\omega(t)$ are $1-nt$ with multiplicity $n-1$ and $1$ with multiplicity $1$. Therefore
$$\det(\omega(t))=\frac{(1-nt)^{n-1}}{r^{2n}},$$
from which it follows that
$\Ric(\omega(t))=\Ric(\omega_H),$ which implies that $\omega(t)$ solves the Chern-Ricci flow
on the maximal existence interval $[0,\frac{1}{n}).$
  \end{proof}

One can also consider more general Hopf manifolds, such as the Hopf surface $M_\alpha$ with $|\alpha_1|\neq|\alpha_2|$. In this case, Gauduchon and Ornea \cite{GO} have constructed an explicit Gauduchon metric $\omega_{\mathrm{GO}}$ (which is also locally conformally K\"ahler). It would be interesting to see if the solution of the Chern-Ricci flow starting at $\omega_{\mathrm{GO}}$ can also be written down explicitly.

Next we give the proof of Proposition \ref{prop:hopf2}.

\begin{proof}[Proof of Proposition \ref{prop:hopf2}]
Write $ \hat{\omega}_t = \omega_H - t \textrm{Ric}(\omega_H)$.  Then we can write $\omega(t)$ as
 $\omega(t) = \hat{\omega}_t + \ddbar \varphi$ for a function $\varphi= \varphi(t)$ solving the parabolic complex Monge-Amp\`ere equation
 \begin{equation}
\ddt{\varphi} = \log \frac{(\hat{\omega}_t + \ddbar \varphi)^n}{\Omega}, \quad \hat{\omega}_t + \ddbar \varphi >0, \quad \varphi|_{t=0}=\psi,
\end{equation}
for $\psi$ as in the statement of Proposition \ref{prop:hopf2}.  A solution exists for $t \in [0,1/n)$.
 In what follows, we will drop the subscript $t$ and write $\hat{\omega}$ for $\hat{\omega}_t$.

 We wish to bound $\tr{\omega_H}{\omega}$ from above.  First we claim that
 \begin{align} \nonumber
 \left( \ddt{} - \Delta \right) \tr{\omega_H}{\omega} & = -g^{\ov{j} i} (\partial_i \partial_{\ov{j}} g_H^{\ov{\ell} k}) g_{k \ov{\ell}} +
 g_H^{\ov{\ell} k} g^{\ov{j}i} ( \partial_k \partial_{\ov{\ell}} \hat{g}_{i \ov{j}} - \partial_i \partial_{\ov{j}} \hat{g}_{k \ov{\ell}}) \\
 &\ \ \ \ \!- 2 \textrm{Re} (g^{\ov{j} i} (\partial_i g_H^{\ov{\ell}k})(\partial_{\ov{j}} g_{k \ov{\ell}}))  - g_H^{\ov{\ell}k} g^{\ov{j}p} g^{\ov{q}i} (\partial_k g_{p\ov{q}})( \partial_{\ov{\ell}} g_{i \ov{j}}).\label{hopfclaim1}
 \end{align}

To see (\ref{hopfclaim1}),  compute
\begin{align} \nonumber
\Delta \tr{\omega_H}{\omega} & = g^{\ov{j} i} \partial_i \partial_{\ov{j}} ( g_H^{\ov{\ell}k} g_{k \ov{\ell}}) \\
& = g^{\ov{j} i} (\partial_i \partial_{\ov{j}} g_H^{\ov{\ell} k}) g_{k \ov{\ell}} + g^{\ov{j}i} g_H^{\ov{\ell} k} \partial_i \partial_{\ov{j}} g_{k \ov{\ell}} + 2\textrm{Re}(g^{\ov{j} i} (\partial_i g_H^{\ov{\ell}k}) (\partial_{\ov{j}} g_{k \ov{\ell}}) ), \label{deltatr}
\end{align}
and
\begin{align} \nonumber
\ddt{} \tr{\omega_H}{\omega} & = g_H^{\ov{\ell}k} \partial_k \partial_{\ov{\ell}} \log \det g \\ \label{ddttr}
& = - g_H^{\ov{\ell}k} g^{\ov{j}p} g^{\ov{q}i} (\partial_k g_{p\ov{q}})( \partial_{\ov{\ell}} g_{i \ov{j}}) + g_H^{\ov{\ell}k} g^{\ov{j} i} \partial_k \partial_{\ov{\ell}} g_{i \ov{j}}.
\end{align}
Then (\ref{hopfclaim1}) follows from combining (\ref{deltatr}) and (\ref{ddttr}) and using the fact that $g_{i\ov{j}} = \hat{g}_{i\ov{j}} + \varphi_{i \ov{j}}$.

For the first term on the right hand side of (\ref{hopfclaim1}), note that
\begin{equation} \nonumber
g_H^{\ov{\ell}k}=r^2\delta_{k\ell},\quad \de_{\ov{j}}g_H^{\ov{\ell}k}=z_j\delta_{k\ell},
\quad \de_i\de_{\ov{j}}g_H^{\ov{\ell}k}=\delta_{ij}\delta_{k\ell},
\end{equation}
so that
\begin{equation} \label{doubletrace}
g^{\ov{j} i} (\partial_i \partial_{\ov{j}} g_H^{\ov{\ell}k}) g_{k \ov{\ell}}=\sum_{i,k} g^{\ov{i}i} g_{k\ov{k}}=
(\tr{\omega_H}{\omega}) (\tr{\omega}{\omega_H}).
\end{equation}

 For the second term on the right hand side of (\ref{hopfclaim1}), calculate
\begin{equation} \label{hopfclaim2}
g_H^{\ov{\ell} k} g^{\ov{j}i}  ( \partial_k \partial_{\ov{\ell}} \hat{g}_{i \ov{j}}- \partial_i \partial_{\ov{j}} \hat{g}_{k \ov{\ell}})
=\tr{\omega}{\Ric(\omega_H})-ng^{\ov{j}i}\frac{\ov{z}_i z_j}{r^4}-(n-2)\tr{\omega}{\omega_H}.
\end{equation}
Indeed, to see (\ref{hopfclaim2}) we
compute
\begin{align} \nonumber
\hat{g}_{i \ov{j}} & = \frac{1}{r^2} \left( (1-nt) \delta_{ij} +nt \frac{\ov{z}_i z_j}{r^2} \right) \\ \label{useful}
\partial_{\ov{\ell}} \hat{g}_{i \ov{j}} & =
- \frac{1}{r^4} z_{\ell} \left( (1-nt) \delta_{ij} + \frac{2nt \ov{z}_i z_j}{r^2} \right) + \frac{nt z_j \delta_{i\ell}}{r^4},
\end{align}
and
\begin{align} \nonumber
\partial_k \partial_{\ov{\ell}} \hat{g}_{i \ov{j}} &= \frac{2}{r^6} \ov{z}_k z_{\ell} \left( (1-nt) \delta_{ij} + \frac{2nt \ov{z}_i z_j}{r^2} \right) - \frac{1}{r^4} \delta_{k \ell} \left((1-nt) \delta_{ij} + \frac{2nt \ov{z}_i z_j}{r^2} \right) \\ \nonumber
& \ \ \ - \frac{1}{r^4} z_{\ell} \left( - \frac{2nt \ov{z}_k \, \ov{z}_i z_j}{r^4} + \frac{2nt \ov{z}_i \delta_{jk}}{r^2} \right) - \frac{2nt \ov{z}_k z_j \delta_{i \ell}}{r^6} + \frac{nt \delta_{jk} \delta_{i\ell}}{r^4} \\ \nonumber
& = \frac{6nt \ov{z}_i z_j \ov{z}_k z_{\ell}}{r^8} + \frac{1}{r^4} (nt \delta_{jk} \delta_{i\ell} - (1-nt) \delta_{k \ell} \delta_{ij}) + \frac{2}{r^6} \ov{z}_k z_{\ell} \delta_{ij} \\ \nonumber
& \ \ \   - \frac{2nt}{r^6} \left( \delta_{k\ell} \ov{z}_i z_j + \delta_{jk} \ov{z}_i z_{\ell} + \delta_{i\ell} \ov{z}_k z_j + \delta_{ij} \ov{z}_k z_{\ell} \right),
\end{align}
Finally, this gives:
$$\partial_k \partial_{\ov{\ell}} \hat{g}_{i \ov{j}} - \partial_i \partial_{\ov{j}} \hat{g}_{k \ov{\ell}} = \frac{2}{r^6} ( \ov{z}_k z_{\ell} \delta_{ij} - \ov{z}_i z_j \delta_{k \ell} ),$$
and
\[\begin{split}
g^{\ov{j}i} g_H^{\ov{\ell} k} ( \partial_k \partial_{\ov{\ell}} \hat{g}_{i \ov{j}}- \partial_i \partial_{\ov{j}} \hat{g}_{k \ov{\ell}}) &= \frac{2}{r^2} g^{\ov{j} i} \left( \delta_{ij} - \frac{n \ov{z}_i z_j}{r^2} \right)\\
&=\tr{\omega}{\Ric(\omega_H})-ng^{\ov{j}i}\frac{\ov{z}_i z_j}{r^4}-(n-2)\tr{\omega}{\omega_H},
\end{split}\]
establishing (\ref{hopfclaim2}).

 Combining (\ref{hopfclaim1}), (\ref{doubletrace}) and (\ref{hopfclaim2}) we obtain
 \begin{align} \nonumber
 \left( \ddt{} - \Delta \right) \tr{\omega_H}{\omega} & = -(\tr{\omega_H}{\omega}) (\tr{\omega}{\omega_H}) + \tr{\omega}{\Ric(\omega_H})-ng^{\ov{j}i}\frac{\ov{z}_i z_j}{r^4}\\ \nonumber
 &\ \ \ \ \! -(n-2)\tr{\omega}{\omega_H}
  - 2 \textrm{Re} (g^{\ov{j} i} (\partial_i g_H^{\ov{\ell}k})(\partial_{\ov{j}} g_{k \ov{\ell}})) \\
  &\ \ \ \ \!  - g_H^{\ov{\ell}k} g^{\ov{j}p} g^{\ov{q}i} (\partial_k g_{p\ov{q}})( \partial_{\ov{\ell}} g_{i \ov{j}}). \label{nevolvetr2}
 \end{align}
 The troublesome term is the 5th one on the right hand side.  We write this term as:
\begin{align} \nonumber
-2\textrm{Re}(g^{\ov{j} i} (\partial_i g_H^{\ov{\ell}k}) (\partial_{\ov{j}} g_{k \ov{\ell}})) & = -2\textrm{Re}(g^{\ov{j} i} (\partial_i g_H^{\ov{\ell}k}) (\partial_{\ov{\ell}} g_{k \ov{j}})) \\ \nonumber
&\ \ \ \ \! - 2\textrm{Re}(g^{\ov{j} i} (\partial_i g_H^{\ov{\ell}k}) (  \partial_{\ov{j}} \hat{g}_{k \ov{\ell}} - \partial_{\ov{\ell}} \hat{g}_{k \ov{j}}) \\ \label{IandII}
& =: A_1 + A_2.
\end{align}
For $A_2$ use \eqref{useful} to compute
\[
\partial_{\ov{j}} \hat{g}_{k \ov{\ell}} - \partial_{\ov{\ell}} \hat{g}_{k \ov{j}}=\frac{1}{r^4} (z_{\ell} \delta_{kj} - z_j \delta_{k\ell}),
\]
and so
\begin{align} \nonumber
A_2 & =- \frac{1}{r^4} 2\textrm{Re} ( g^{\ov{j}i} \ov{z}_i \delta_{k \ell} (z_{\ell} \delta_{kj} - z_j \delta_{k\ell}) )\\ \nonumber
& =- \frac{1}{r^4} 2 \textrm{Re} \left( g^{\ov{j}i} (\ov{z}_i z_j- n \ov{z}_i z_j) \right)\\ \label{A2}
& =   2(n-1)  g^{\ov{j}i} \frac{\ov{z}_i z_j}{r^4}.
\end{align}

To deal with $A_1$ we introduce an inner product on tensors of type $\Phi = \Phi_{i j \ov{k}}$.  For tensors $\Psi$ and $\Phi$ of this type, define
$$\langle \Psi, \Phi \rangle = g_H^{\ov{j}i} g^{\ov{\ell}k} g^{\ov{q}p} \Psi_{i k \ov{q}} \ov{ \Phi_{j \ell \ov{p}}}.$$
Then if $\Phi_{i j \ov{k}} = \partial_i g_{j\ov{k}}$ we see that the last term on the right hand side of (\ref{nevolvetr2}) is $- | \Phi|^2$.

Now compute
\begin{align} \nonumber
A_1 & = -2\textrm{Re}(g^{\ov{j} i} (\partial_i g_H^{\ov{\ell}k}) (\partial_{\ov{\ell}} g_{k \ov{j}})) \\ \nonumber
& = - 2 \textrm{Re} (g^{\ov{j} i } g_H^{\ov{\ell} p} g^{\ov{v}k} g_{u \ov{v}} (g_H)_{p \ov{q}} (\partial_i g_H^{\ov{q} u}) \partial_{\ov{\ell}} g_{k \ov{j}})\\ \nonumber
& = - 2 \textrm{Re} \langle \Psi, \Phi \rangle \\
& \le 2 |\Psi| |\Phi|,
\end{align}
with $\Psi_{i j \ov{k}} = (g_H)_{i \ov{q}} (\partial_j g_H^{\ov{q}u})  g_{u \ov{k}}$ and  $\Phi_{i j \ov{k}} = \partial_i g_{j\ov{k}}$.  But
$$\Psi_{i j \ov{k}} = \frac{1}{r^2} \delta_{iq}  \ov{z}_j \delta_{uq} g_{u\ov{k}} = \frac{1}{r^2} \ov{z}_j g_{i \ov{k}},$$
and so
$$|\Psi|^2 = \frac{1}{r^4} g_H^{\ov{j} i} g^{\ov{\ell} k} g^{\ov{q}p} \ov{z}_k g_{i \ov{q}} z_{\ell} g_{p \ov{j}} = \frac{1}{r^4} g_H^{\ov{j}i} g_{i \ov{j}} g^{\ov{\ell} k} \ov{z}_k z_{\ell} =( \tr{\omega_H}{\omega}) g^{\ov{\ell} k} \frac{\ov{z}_k z_{\ell}}{r^4}. $$
Then
\begin{align} \nonumber
A_1 & \le |\Psi|^2+ |\Phi|^2 \\ \nonumber
& = ( \tr{\omega_H}{\omega}) g^{\ov{\ell} k} \frac{\ov{z}_k z_{\ell}}{r^4} + |\Phi|^2 \\
& = (\tr{\omega_H}{\omega}) (\tr{\omega}{\omega_H}) - (\tr{\omega_H}{\omega}) \frac{1}{r^2} g^{\ov{j}i} \left( \delta_{ij} - \frac{\ov{z}_i z_j}{r^2} \right) + |\Phi|^2. \label{I}
\end{align}
Combining (\ref{nevolvetr2}), (\ref{IandII}), (\ref{A2}) and (\ref{I}) we have
\[\begin{split}
\left(\ddt{} - \Delta\right) \tr{\omega_H}{\omega} & \le \tr{\omega}{\Ric(\omega_H)}
-ng^{\ov{j}i}\frac{\ov{z}_i z_j}{r^4}-(n-2)\tr{\omega}{\omega_H}\\
&\ \ \ - (\tr{\omega_H}{\omega}) \frac{1}{r^2} g^{\ov{j}i} \left( \delta_{ij} - \frac{\ov{z}_i z_j}{r^2} \right)
+ 2(n-1)  g^{\ov{j}i} \frac{\ov{z}_i z_j}{r^4}\\
&=\tr{\omega}{\Ric(\omega_H)}-(n-2)\frac{1}{r^2}g^{\ov{j}i} \left( \delta_{ij} - \frac{\ov{z}_i z_j}{r^2} \right)\\
&\ \ \ - \frac{1}{n}(\tr{\omega_H}{\omega})\tr{\omega}{\Ric(\omega_H)} \\
&=\tr{\omega}{\Ric(\omega_H)}-\frac{n-2}{n}\tr{\omega}{\Ric(\omega_H)}\\
&\ \ \ - \frac{1}{n}(\tr{\omega_H}{\omega})\tr{\omega}{\Ric(\omega_H)} \\
& = \left(\frac{2}{n} - \frac{1}{n} \tr{\omega_H}{\omega}\right) \tr{\omega}{\textrm{Ric}(\omega_H)}.
\end{split}\]
Since $\Ric(\omega_H)\ge 0$, we
 conclude by the maximum principle that $\tr{\omega_H}{\omega}$ is uniformly bounded from above.
\end{proof}

Finally, we remark that this implies the convergence of the Chern-Ricci flow at the level of potentials.  Indeed, recall that $\omega(t) = \hat{\omega}_t + \ddbar \varphi$ with $\varphi$ solving (\ref{ma2}) and $\hat{\omega}_t = \omega_H - t \Ric(\omega_H)$.  Moreover, $\hat{\omega}_T$ is nonnegative, so we can apply the argument of Proposition \ref{prop:nn} to obtain uniform upper bounds for $|\varphi|$ and $\dot{\varphi}$.  It follows immediately that as $t \rightarrow T$, $\varphi(t)$ converges pointwise to a function $\varphi(T)$ on $M_{\alpha}$.
On the other hand, by Proposition \ref{prop:hopf2} we have uniform bounds  for $|\Delta_{g_H} \varphi|$.  Thus standard elliptic theory gives a uniform bound for $\| \varphi(t) \|_{C^{1+\beta}}$  for any $\beta \in (0,1)$.  It follows that $\varphi \rightarrow \varphi(T)$ in $C^{1+\beta}$  for any $\beta \in (0,1)$.

\section{The complex Monge-Amp\`ere equation} \label{section:monge}

In this section we prove uniform estimates for solutions of the elliptic complex Monge-Amp\`ere equation.

Let $(M^n,\omega)$ be a compact Hermitian manifold, $F$ a smooth function on $M$ and
$\omega'=\omega+\ddbar\vp$ a Hermitian metric that satisfies
\begin{equation}\label{maell2}
(\omega+\ddbar\vp)^n=e^F\omega^n.
\end{equation}
We will give an alternative proof of the main result of \cite{TW2} (see also \cite{Bl, DK} for different proofs):

\begin{theorem} \label{theorem:tw}
There is a constant $C$ that depends only on $(M,\omega)$, $\sup_M F$ and $\inf_M \Delta F$ such that
\begin{equation}\label{c0}
\sup_M \vp-\inf_M\vp\leq C.
\end{equation}
\end{theorem}
This is slightly weaker than the result in \cite{TW2}, because there is no dependence of $C$ on $\inf_M \Delta F$ there (and in fact the proof of \cite{TW2} can be easily modified to have $C$ depend only on $p>n$ and $\int_M e^{pF}$ rather than $\sup_M F$).  The point of our discussion here is to establish Theorem \ref{theorem:tw} via a new
 second order estimate which we had previously established \cite{TW1} using the maximum principle only in the cases $n=2$ or $(M, \omega)$ balanced.

From now on we will normalize $\vp$ by assuming $\sup_M\vp=0$.
The estimate we wish to prove is:
\begin{equation}\label{c2}
\tr{g}{g'}\leq Ce^{A(\vp-\inf_M\vp)}
\end{equation}
for uniform constants $C,A$.  The reader may notice that (\ref{c2}) has  the same form as the second order estimates of Yau and Aubin \cite{Au, Y}.

Theorem \ref{theorem:tw} then follows from (9.3).  Indeed,
we can then use the arguments in \cite{TW1}   to derive \eqref{c0} from \eqref{c2}.   The idea is that a second order estimate of the form (\ref{c2}), together with the condition $\ddbar \varphi \ge -  \omega$
 implies, via a Moser iteration argument applied to the exponential of $\varphi$, a zero order estimate for $\varphi$.  This method was employed  in the K\"ahler case in \cite{W}, and a related argument was used in the almost complex setting in \cite{TWY}.

\begin{proof}[Proof of \eqref{c2}]
Following Phong-Sturm \cite{PS} we consider the quantity
$$Q=\log \tr{g}{g'}-A\vp+\frac{1}{\vp-\inf_M\vp +1},$$
for $A\geq 1$ to be determined. Note that $0<\frac{1}{\vp-\inf_M\vp +1}\leq 1$.
We have
\begin{equation}\label{step1}
\begin{split}
\Delta'Q&=\Delta'\log \tr{g}{g'} -An+A\tr{g'}{g}
-\frac{n-\tr{g'}{g}}{(\vp-\inf_M\vp +1)^2} \\ & \ \ \ +\frac{2|\de\vp|^2_{g'}}{(\vp-\inf_M\vp +1)^3}\\
&\geq \Delta'\log \tr{g}{g'} +A\tr{g'}{g} +\frac{2|\de\vp|^2_{g'}}{(\vp-\inf_M\vp +1)^3}
-An-n,
\end{split}
\end{equation}
writing $\Delta'$ for the complex Laplacian associated to $g'$.
To calculate $\Delta'\log \tr{g}{g'}$ we go back to the calculations in Proposition \ref{prop:bigcalc}
where $\hat{g}=g_0$ is now replaced by $g$ and $g'$ takes the role of the evolving metric there. With these
substitutions \eqref{laplac}, \eqref{laplac2} and \eqref{later} together read
\[
\begin{split}
\Delta' \tr{g}g'&=g'^{\ov{j}i}g^{\ov{\ell}k}\nabla_{\ov{\ell}}\nabla_k g'_{i\ov{j}}
+g'^{\ov{j}i} \nabla_i \ov{T^\ell_{j\ell}}+g'^{\ov{j}i}g^{\ov{\ell}k} g_{p \ov{j}}\nabla_{\ov{\ell}} T^p_{ik} \\
&\ \ \  -g'^{\ov{j}i}g^{\ov{\ell}k}g'_{k\ov{q}} (\nabla_i\ov{T^q_{j\ell}}-R_{i\ov{\ell}p\ov{j}}g^{\ov{q}p})\\
&\ \ \ -R
-2\mathrm{Re}\left(g'^{\ov{j}i}g^{\ov{\ell}k} T^p_{ik}\nabla_{\ov{\ell}}g'_{p\ov{j}}\right)\\
&\ \ \ -g'^{\ov{j}i}g^{\ov{\ell}k}T^p_{ik}\ov{T^q_{j\ell}}g_{p\ov{q}}
+g'^{\ov{j}i}g^{\ov{\ell}k}T^p_{ik}\ov{T^q_{j\ell}}g'_{p\ov{q}},
\end{split}\]
where $R=g^{\ov{\ell}k}\RC_{k\ov{\ell}}=g^{\ov{j}i}g^{\ov{\ell}k}R_{k\ov{\ell}i\ov{j}}$ is the Chern scalar curvature of $g$.
On the other hand by applying $\Delta\log$ to \eqref{maell2} we get
$$\Delta F-R=g^{\ov{\ell}k}\de_k\de_{\ov{\ell}}
\log\det(g')=g'^{\ov{j}i}g^{\ov{\ell}k}\de_k\de_{\ov{\ell}} g'_{i\ov{j}}
-g'^{\ov{j}p}g'^{\ov{q}i}g^{\ov{\ell}k}\de_k g'_{i\ov{j}}\de_{\ov{\ell}}g'_{p\ov{q}},$$
and converting these into covariant derivatives (as in the argument for (\ref{ddttrace}) in the
proof of Proposition \ref{prop:bigcalc}) we get
$$\Delta F=g'^{\ov{j}i}g^{\ov{\ell}k}\nabla_{\ov{\ell}}\nabla_k g'_{i\ov{j}}
-g'^{\ov{j}p}g'^{\ov{q}i}g^{\ov{\ell}k}\nabla_k g'_{i\ov{j}}\nabla_{\ov{\ell}}g'_{p\ov{q}},$$
and so
\[
\begin{split}
\Delta' \log\tr{g}g'&=\frac{1}{\tr{g}g'}\Bigg(\bigg[g'^{\ov{j}p}g'^{\ov{q}i}g^{\ov{\ell}k}\nabla_k g'_{i\ov{j}}\nabla_{\ov{\ell}}g'_{p\ov{q}}
-\frac{1}{\tr{g}g'} g'^{\ov{\ell}k}\nabla_k \tr{g}{g'}\nabla_{\ov{\ell}}\tr{g}{g'}\\
&\ \ \ +2\mathrm{Re}\left(g'^{\ov{j}i}g^{\ov{\ell}k} T^p_{ki}\nabla_{\ov{\ell}}g'_{p\ov{j}}\right)
+g'^{\ov{j}i}g^{\ov{\ell}k}T^p_{ik}\ov{T^q_{j\ell}}g'_{p\ov{q}}\bigg]
+\Delta F- R\\
&\ \ \ +g'^{\ov{j}i} \nabla_i \ov{T^\ell_{j\ell}}+g'^{\ov{j}i}g^{\ov{\ell}k} g_{p \ov{j}}\nabla_{\ov{\ell}} T^p_{ik} -g'^{\ov{j}i}g^{\ov{\ell}k}g'_{k\ov{q}} (\nabla_i\ov{T^q_{j\ell}}
-R_{i\ov{\ell}p\ov{j}}g^{\ov{q}p})\\
&\ \ \ -g'^{\ov{j}i}g^{\ov{\ell}k}T^p_{ik}\ov{T^q_{j\ell}}g_{p\ov{q}}
\Bigg).
\end{split}\]
The Cauchy-Schwarz argument from \eqref{later2} shows that the quantity inside square brackets equals
$$\bigg[\cdots\bigg]=K+2\mathrm{Re}\left(
g'^{\ov{q}k} T^i_{ik} \frac{\nabla_{\ov{q}} \tr{g}g'}{\tr{g}g'}\right),$$
and $K\geq 0$ is the same quantity as in \eqref{later2}. Putting these together we have
\begin{equation}\label{step2}
\Delta' \log\tr{g}g'\geq \frac{2}{(\tr{g}g')^2}\mathrm{Re}\left(g'^{\ov{q}k} T^i_{ik} \nabla_{\ov{q}} \tr{g}g' \right)-C\tr{g'}g-C,
\end{equation}
where we used the fact that $\tr{g}g'\geq C^{-1}$ which is a simple consequence of \eqref{maell2} and the arithmetic-geometric means inequality.  Suppose that $Q$ achieves its maximum at a point $x \in M$.  Then at $x$ we have $\de_i Q=0$ and hence
$$\frac{1}{\tr{g}{g'}}\de_i\tr{g}{g'} -A\de_i\vp -\frac{1}{(\vp-\inf_M\vp +1)^2}\de_i\vp=0,$$
therefore
\[\begin{split}
&\left|\frac{2}{(\tr{g}g')^2}\mathrm{Re}\left(g'^{\ov{q}k} T^i_{ik} \nabla_{\ov{q}} \tr{g}g' \right)\right|
\\
&=\left| \frac{2}{\tr{g}{g'}}
\textrm{Re} \left(\left(A+\frac{1}{(\vp-\inf_M\vp +1)^2} \right)g'^{\ov{q}k} T^i_{ik} \nabla_{\ov{q}} \vp\right) \right|\\
&\leq \frac{|\de\vp|^2_{g'}}{(\vp-\inf_M\vp +1)^3}+CA^2(\vp-\inf_M\vp +1)^3\frac{\tr{g'}{g}}{(\tr{g}{g'})^2}.
\end{split}\]
If at  $x$ we have $(\tr{g}{g'})^2(x)\leq A^2(\vp(x)-\inf_M\vp +1)^3$ then
we also have (using that $\frac{3}{2}\log t\leq t$ for $t>0$)
\[\begin{split}
Q\leq Q(x)&\leq \frac{3}{2}\log (\vp(x)-\inf_M\vp +1)+\log A-A\vp(x)+1\\
&\leq -(A-1)\vp(x)-\inf_M\vp +\log A+2\leq -A\inf_M\vp+\log A+2,
\end{split}\]
and so in this case \eqref{c2} follows immediately.

Otherwise, we have  $(\tr{g}{g'})^2(x)\geq A^2(\vp(x)-\inf_M\vp +1)^3$ and so at $x$,
\begin{equation}\label{step3}
\left|\frac{2}{(\tr{g}g')^2}\mathrm{Re}\left(g'^{\ov{q}k} T^i_{ik} \nabla_{\ov{q}} \tr{g}g' \right)\right|
\leq \frac{|\de\vp|^2_{g'}}{(\vp-\inf_M\vp +1)^3}+C\tr{g'}{g}.
\end{equation}
Combining \eqref{step1}, \eqref{step2} and \eqref{step3} we get, at $x$,
$$0\geq \Delta'Q\geq A\tr{g'}{g}-C\tr{g'}{g}-An-n-C\geq \tr{g'}{g}-C,$$
if $A$ is chosen sufficiently large. From this \eqref{c2} follows easily.
\end{proof}

\end{document}